\documentclass[a4paper,11pt]{article}

\usepackage[T1]{fontenc}
\usepackage{lmodern}

\usepackage{dsfont}

\usepackage[latin1]{inputenc}
\usepackage{amsmath}
\usepackage{amsthm}
\usepackage{amssymb}
\usepackage{mathrsfs}
\usepackage{graphicx}
\usepackage[all]{xy}
\usepackage{hyperref}

\usepackage{stmaryrd}
\usepackage{caption}

\usepackage{abstract} 

\newtheorem{thm}{Theorem}[section]
\newtheorem{cor}[thm]{Corollary}
\newtheorem{claim}[thm]{Claim}
\newtheorem{fact}[thm]{Fact}

\newtheorem{lemma}[thm]{Lemma}
\newtheorem{prop}[thm]{Proposition}

\theoremstyle{definition}
\newtheorem{definition}[thm]{Definition}
\newtheorem{ex}[thm]{Example}
\newtheorem{remark}[thm]{Remark}
\newtheorem{question}[thm]{Question}

\newtheorem{notation}[thm]{Notation}

\newcommand{\QM}{\mathrm{QM}(\Gamma, \mathcal{G})}

\title{Special cube complexes revisited: a quasi-median generalisation}
\date{\today}
\author{Anthony Genevois}

\begin{document}

\maketitle

\begin{abstract}
In this article, we generalise Haglund and Wise's theory of special cube complexes to groups acting on quasi-median graphs. More precisely, we define special actions on quasi-median graphs, and we show that a group which acts specially on a quasi-median graph with finitely many orbits of vertices must embed as a virtual retract into a graph product of finite extensions of clique-stabilisers. In the second part of the article, we apply the theory to fundamental groups of some graphs of groups called \emph{right-angled graphs of groups}. 
\end{abstract}

\tableofcontents

\section{Introduction}

\noindent
Haglund and Wise's theory of special cube complexes \cite{MR2377497} is one of the major contributions of the study of groups acting on CAT(0) cube complexes. The key point of the theory is that, if a group $G$ can be described as the fundamental group of a nonpositively curved cube complex $X$, then there exists a simple and natural condition about $X$ which implies that $G$ can be embedded into a right-angled Artin group $A$. As a consequence, all the properties which are satisfied by right-angled Artin groups and which are stable under taking subgroups are automatically satisfied by our group $G$, providing valuable information about it. Such properties include: 
\begin{itemize}
	\item two-generated subgroups are either free abelian or free non-abelian \cite{MR634562};
	\item any subgroup either is free abelian or surjects onto $\mathbb{F}_2$ \cite[Corollary 1.6]{MR3365774};
	\item being bi-orderable \cite{MR1179860, MR1189240};
	\item being linear (and, in particular, residually finite) \cite{MR1302120}; 
	\item being residually torsion-free nilpotent \cite{DromsThesis, MR1179860, MR3539842}.
\end{itemize}
Even better, as soon as the cube complex $X$ is compact, the theory does not only show that $G$ embeds into $A$, it shows that it embeds in a very specific way: the image of $G$ in $A$ is a \emph{virtual retract}, i.e., there exists a finite-index subgroup $H \leq A$ containing $G$ and a morphism $r : H \to G$ such that $r_{|G}= \mathrm{Id}_G$. This additional information provides other automatic properties satisfied by our group, including:
\begin{itemize}
	\item two-generated subgroups are undistorted \cite{TwoGenUndistorted};
	\item infinite cyclic subgroups are separable \cite{MR2427635};
	\item being conjugacy separable \cite{MR2914863}.
\end{itemize}
One of the most impressive application of the theory of special cube complexes is Agol's proof of the virtual Haken conjecture \cite{MR3104553}, showing that any cubulable hyperbolic group must be \emph{cocompact special}. But the scope of the theory is not restricted to hyperbolic groups and encompasses a large diversity of groups (possibly up to finite index), such as Coxeter groups, many 3-manifold groups, and graph braid groups.

\medskip \noindent
In this article, our goal is to generalise Haglund and Wise's theory by replacing CAT(0) cube complexes with \emph{quasi-median graphs} and right-angled Artin groups with \emph{graph products of groups}. 

\medskip \noindent
As shown in \cite{quasimedian}, quasi-median graphs, a family of graphs generalising median graphs (or equivalently, CAT(0) cube complexes), have a long history in metric graph theory. In \cite{Qm}, we introduced them in geometric group theory by showing how they can be exploited in the study of graph products of groups, lamplighter groups and Thompson-like groups (see also \cite{MR4033512}). It turned out that quasi-median graphs provide a particularly relevant point of view in order to study graph products of groups \cite{AutGP, AutRAAG, EccentricGP}. 

\medskip \noindent
Recall from \cite{GreenGP} that, given a simplicial graph $\Gamma$ and a collection of groups $\mathcal{G}= \{G_u \mid u \in V(\Gamma) \}$ indexed by the vertices of $\Gamma$, the \emph{graph product} $\Gamma \mathcal{G}$ is the quotient
$$\left( \underset{u \in V(\Gamma)}{\ast} G_u \right) / \langle \langle [g,h]=1, \ g \in G_u, h\in G_v, \{u,v\} \in E(\Gamma) \rangle \rangle$$
where $V(\Gamma)$ and $E(\Gamma)$ denote the vertex- and edge-sets of $\Gamma$. For instance, if the groups in $\mathcal{G}$ are all infinite cyclic, then $\Gamma \mathcal{G}$ coincides with the right-angled Artin group $A_\Gamma$; and if all the groups in $\mathcal{G}$ are cyclic of order two, then $\Gamma \mathcal{G}$ coincides with the right-angled Coxeter group $C_\Gamma$. In the same way that the Cayley graphs of $A_\Gamma$ and $C_\Gamma$ constructed from the generating set $V(\Gamma)$ are median graphs (or equivalently, that their cube completions are CAT(0) cube complexes), the Cayley graph
$$\QM := \mathrm{Cayl} \left(\Gamma \mathcal{G},  \bigcup\limits_{u \in V(\Gamma)} G_u \backslash \{1\} \right)$$
of $\Gamma \mathcal{G}$ turns out to be a quasi-median graph. 

\medskip \noindent
So, given a group $G$ acting on a quasi-median graph $X$, we want to identify a simple condition on the action $G \curvearrowright X$ which implies that $G$ naturally embeds into a graph product, possibly as a virtual retract. As shown in Sections \ref{section:warmup} and \ref{section:embedding}, the following definition includes naturally the groups considered in Haglund and Wise's theory.

\begin{definition}
Let $G$ be a group acting faithfully on a quasi-median graph $X$. The action is \emph{hyperplane-special} if
\begin{itemize}
	\item for every hyperplane $J$ and every element $g \in G$, $J$ and $gJ$ are neither transverse nor tangent;
	\item for all hyperplanes $J_1,J_2$ and every element $g \in G$, if $J_1$ and $J_2$ are transverse then $J_1$ and $gJ_2$ cannot be tangent.
\end{itemize}
The action is \emph{special} if, in addition, the action $\mathfrak{S}(J) \curvearrowright \mathscr{S}(J)$ is~free for every hyperplane~$J$ of $X$. (Here, $\mathscr{S}(J)$ denotes the collection of all the sectors delimited by $J$, i.e., the connected components of the graph obtained from $X$ by removing the interiors of all the edges dual to $J$; and $\mathfrak{S}(J)$ denotes the image of $\mathrm{stab}(J)$ in the permutation group of $\mathscr{S}(J)$.)
\end{definition}

\noindent
The main result of this article is the following embedding theorem. (We refer to Theorems \ref{thm:IntroEmbedding} and \ref{thm:IntroRetract} for more precise statements.)

\begin{thm}\label{thm:IntroBig}
Let $G$ be a group which acts specially on a quasi-median graph with finitely many orbits of vertices. Then $G$ embeds as a virtual retract into a graph product of finite extensions of clique-stabilisers. 
\end{thm}

\noindent
As in Haglund and Wise's theory, knowing that the group we are studying is a subgroup of a graph product provides valuable information about it. For instance:

\begin{cor}
Let $G$ be a group which acts specially on a quasi-median graph with finitely many orbits of vertices. Then the following assertions hold.
\begin{itemize}
	\item Assume that clique-stabilisers satisfies Tits' alternative, i.e., every subgroup either contains a non-abelian free subgroup or is virtually solvable. Then $G$ also satisfies Tits' alternative. \cite{MR3365774}
	\item If clique-stabilisers are linear (resp. residually finite), then so is $G$. \cite{GPlinear, GreenGP}
	\item If clique-stabilisers are (bi-)orderable (resp. locally indicable), then so is $G$. \cite{MR2946302, MR3365774}
	\item If clique-stabilisers are a-T-menable (resp. weakly amenable), then so is $G$. \cite{HaagerupGP, Qm, MR3687943}
\end{itemize}
\end{cor}

\noindent
The fact that the image of our embedding is a virtual retract also provides additional information:

\begin{cor}
Let $G$ be a group which acts specially on a quasi-median graph with finitely many orbits of vertices. Then the following assertions hold.
\begin{itemize}
	\item For every $n \geq 1$, if clique-stabilisers are of type $F_n$, then so is $G$.\cite{MR1293049, MR1317337, MR1377652} In particular, if clique-stabilisers are finitely generated (resp. finitely presented), then so is $G$.
	\item If clique-stabilisers are finitely presented, then the coarse inequality $$\delta_G \prec \max(n \mapsto n^2, \delta_{\mathrm{stab}(C)}, \text{$C$ clique})$$ between Dehn functions holds.\cite{MR1082628, MR1348572, MR1377652, VanKampenGP}
	\item If clique-stabilisers are conjugacy separable, then so is $G$. \cite{MR3427632} 
	\item If clique-stabilisers have their cyclic subgroups separable, then cyclic subgroups of $G$ are separable. \cite{MR3950469}
	\item If clique-stabilisers are finitely generated and have their infinite cyclic subgroups undistorted, then infinite cyclic subgroups in $G$ are undistorted. 
\end{itemize}
\end{cor}

\paragraph{A word about the proof of the theorem.} In Section \ref{section:warmup}, we explain how the fundamental group $G$ of a special cube complex $X$ can be embedded into a right-angled Artin group by looking at the action of $G$ on the universal cover of $X$, instead of looking for a local isometry of $X$ to the Salvetti complex of a right-angled Artin group. This construction is next generalised to arbitrary quasi-median graphs in Section \ref{section:embedding} in order to prove:

\begin{thm}\label{thm:IntroEmbedding}
Let $G$ be a group acting specially on a quasi-median graph $X$. 
\begin{itemize}
	\item Fix representatives $\{J_i \mid i \in I\}$ of hyperplanes of $X$ modulo the action of $G$. 
	\item Let $\Gamma$ denote the graph whose vertex-set is $\{J_i \mid i\in I\}$ and whose edges link two hyperplanes if they have two transverse $G$-translates. 
	\item For every $i \in I$, let $G_i$ denote the group $\mathfrak{S}(J_i) \oplus K_i$, where $K_i$ is an arbitrary group of cardinality the number of orbits of $\mathfrak{S}(J_i) \curvearrowright \mathscr{S}(J_i)$. 
\end{itemize}
Then there exists an injective morphism $\varphi : G \hookrightarrow \Gamma \mathcal{G}$, where $\mathcal{G}= \{ G_i \mid i \in I\}$, and a $\varphi$-equivariant embedding $X \hookrightarrow \QM$ whose image is gated.
\end{thm}

\noindent
Notice that, compared to Theorem \ref{thm:IntroBig}, we do not require the action to have only finitely many orbits of vertices. Under this additional assumption, we observe in Corollary \ref{cor:VertexGroups} that each $G_i$ contains a clique-stabiliser as a finite-index subgroup, concluding the first step towards the proof of Theorem \ref{thm:IntroBig}.

\medskip \noindent
The next step is to show that the image of our embedding is a virtual retract. The key point is that the image of $X \hookrightarrow \QM$ in Theorem \ref{thm:IntroEmbedding} is \emph{gated}, which is a strong convexity condition. Combined with the next statement, the proof of Theorem~\ref{thm:IntroBig} follows.

\begin{thm}\label{thm:IntroRetract}
Let $\Gamma$ be a simplicial graph and $\mathcal{G}$ a collection of groups indexed by $V(\Gamma)$. A gated-cocompact subgroup $H \leq \Gamma \mathcal{G}$ is a virtual retract.
\end{thm}

\noindent
Here, a subgroup $H \leq \Gamma \mathcal{G}$ is \emph{convex-compact} if there exists a gated subgraph in $\QM$ on which $H$ acts with finitely many orbits of vertices. It is worth noticing that, combined with Theorem \ref{thm:IntroEmbedding}, Theorem \ref{thm:IntroRetract} implies more generally that gated-cocompact subgroups are virtual retracts in arbitrary groups acting specially on quasi-median graphs with finitely many vertices (see Corollary \ref{cor:GatedRetract}), generalising the fact that convex-cocompact subgroups are virtual retracts in cocompact special groups \cite{MR2377497}.

\paragraph{Applications.} In the second part of the article, we apply the theory of groups acting specially on quasi-median graphs to a specific family of groups originated from \cite{Qm}, namely fundamental groups of \emph{right-angled graphs of groups}. We refer to Section \ref{section:IntroRAGG} for a precise definition, but roughly speaking a graph of groups is said right-angled if its vertex-groups are graph products and if its edge-groups are subgraph products. In Section \ref{section:WhenSpecial}, we characterise precisely when the action of the fundamental group of such a graph of groups on the quasi-median graph constructed in \cite{Qm} is special. 

\medskip \noindent
In order to illustrate how special actions on quasi-median graphs can be exploited, let us conclude this introduction by considering an explicit example (detailed in Section \ref{section:RAGGex}). 

\medskip \noindent
Given a group $A$, define $A^\rtimes$ by the relative presentation
$$\langle A,t \mid [a,tat^{-1}]=1, \ a \in A \rangle.$$
Notice that, if $A$ is infinite cyclic, we recover the group introduced in \cite{MR897431}, which was the first example of fundamental group of a 3-manifold which is not subgroup separable. $A^\rtimes$ is an example of a fundamental group of a right-angled graph of groups. It acts on a quasi-median graph, but this action is not special. Such a negative result is not a flaw in the strategy: as a two-generated group which is neither abelian nor free, $\mathbb{Z}^\rtimes$ cannot be embedded into a right-angled Artin group. Nevertheless, considering a finite cover of the graph of groups defining $A^\rtimes$ naturally leads to a new group, denoted by $A \square A$ and admitting
$$\langle A_1,A_2,t \mid [a_1,a_2]=[a_1,ta_2t^{-1}]=1, \ a_1 \in A_1,a_2 \in A_2 \rangle$$
as a relative presentation, where $A_1$ and $A_2$ are two copies of $A$. Then $A \square A$ is a subgroup of $A^\rtimes$ of index two. Now, as the fundamental group of a right-angled graph of groups, $A \square A$ acts specially on a quasi-median graph. By a careful application of Theorem \ref{thm:IntroEmbedding}, we find that $A \square A$ embeds (as a virtual retract) into the graph product
$$G: = \ \mathbb{Z}_2 \ \text{---} \ A_1 \ \text{---} \ A_2 \ \text{---} \ \mathbb{Z}_2$$
by sending $A_1 \subset A \square A$ to $A_1 \subset G$, $A_2 \subset A \square A$ to $A_2 \subset G$, and $t \in A \square A$ to $xy \in G$ where $x$ and $y$ are generators of the two $\mathbb{Z}_2$.

\paragraph{Acknowledgments.} This work was supported by a public grant as part of the Fondation Math\'ematique Jacques Hadamard.

\section{Preliminary}

\noindent
In this section, we give the basic definitions and properties about quasi-median graphs and graph products of groups which will be needed in the rest of the article.

\paragraph{Quasi-median graphs.} There exist several equivalent definitions of quasi-median graphs; see for instance \cite{quasimedian}. Below is the definition used in \cite{Qm}.

\begin{definition}
A connected graph $X$ is \emph{quasi-median} if it does not contain $K_4^-$ and $K_{3,2}$ as induced subgraphs, and if it satisfies the following two conditions:
\begin{description}
	\item[(triangle condition)] for every vertices $a, x,y \in X$, if $x$ and $y$ are adjacent and if $d(a,x)=d(a,y)$, then there exists a vertex $z \in X$ which adjacent to both $x$ and $y$ and which satisfies $d(a,z)=d(a,x)-1$;
	\item[(quadrangle condition)] for every vertices $a,x,y,z \in X$, if $z$ is adjacent to both $x$ and $y$ and if $d(a,x)=d(a,y)=d(a,z)-1$, then there exists a vertex $w \in X$ which adjacent to both $x$ and $y$ and which satisfies $d(a,w)=d(a,z)-2$.
\end{description}
\end{definition}

\noindent
The graph $K_{3,2}$ is the bipartite complete graph, corresponding to two squares glued along two adjacent edges; and $K_4^-$ is the complete graph on four vertices minus an edge, corresponding to two triangles glued along an edge. The triangle and quadrangle conditions are illustrated by Figure \ref{Quadrangle}.
\begin{figure}
\begin{center}
\includegraphics[scale=0.45]{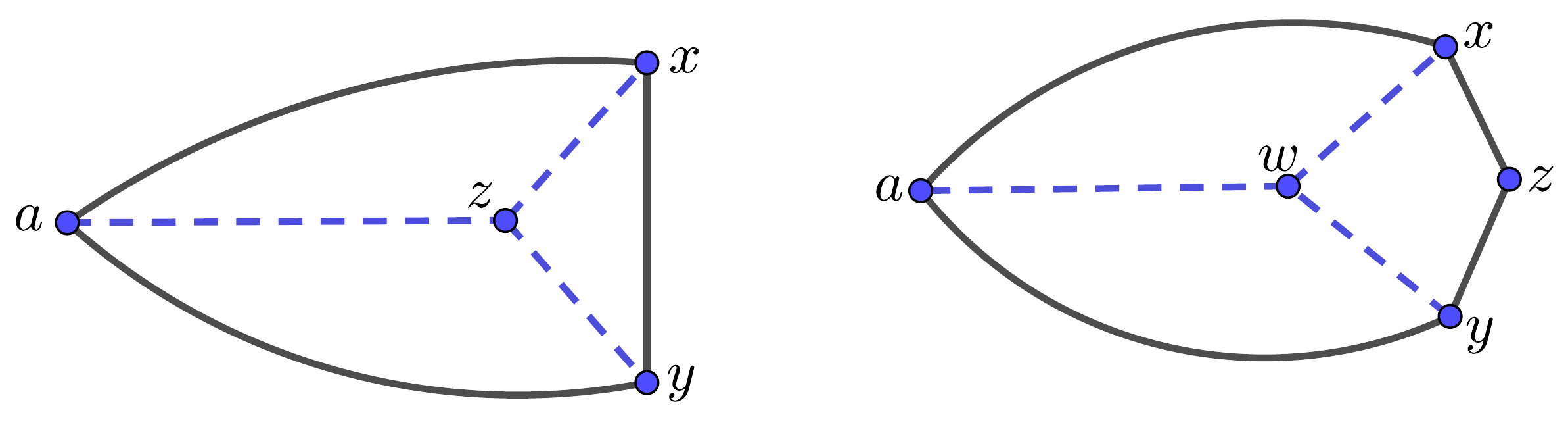}
\caption{Triangle and quadrangle conditions.}
\label{Quadrangle}
\end{center}
\end{figure}

\begin{definition}
Let $X$ be a graph and $Y \subset X$ a subgraph. A vertex $y \in Y$ is a \emph{gate} of an other vertex $x \in X$ if, for every $z \in Y$, there exists a geodesic between $x$ and $z$ passing through $y$. If every vertex of $X$ admits a gate in $Y$, then $Y$ is \emph{gated}.
\end{definition}

\noindent
It is worth noticing that the gate of $x$ in $Y$, when it exists, is unique and minimises the distance to $x$ in $Y$. As a consequence, it may be referred to as the \emph{projection} of $x$ onto $Y$. Gated subgraphs in quasi-median graphs play the role of convex subcomplexes in CAT(0) cube complexes. We record the following useful criterion for future use; a proof can be found in \cite{ChepoiTriangles} (see also \cite[Proposition 2.6]{Qm}).

\begin{lemma}\label{lem:GatedSub}
Let $X$ be a quasi-median graph and $Y \subset X$ a connected induced subgraph. Then $Y$ is gated if and only if it is locally convex (i.e., any 4-cycle in $X$ with two adjacent edges contained in $Y$ necessarily lies in $Y$) and if it contains its triangles (i.e., any 3-cycle which shares an edge with $Y$ necessarily lies in $Y$).
\end{lemma}

\noindent
Recall that a \emph{clique} is a maximal complete subgraph, and that cliques in quasi-median graphs are gated \cite{quasimedian}. A \emph{prism} is a subgraph which a product of cliques. 

\medskip \noindent
By filling in prisms with products of simplices, quasi-median graphs can be thought of as prism complexes. As proved in \cite{Qm} such prism complexes turn out to be CAT(0) spaces, and in particular simply connected. As a consequence, we deduce the following statement which will be useful later:

\begin{lemma}\label{lem:GeodFlipSquare}
Let $X$ be a quasi-median graph, $x,y \in X$ two vertices, and $\gamma_1,\gamma_2$ two paths between $x$ and $y$. Then $\gamma_2$ can be obtained from $\gamma_1$ by flipping squares, shortening triangles, removing backtracks, and inverses of these operations.
\end{lemma}

\noindent
Our lemma requires a few definitions. Given an oriented path $\gamma$ in our graph $X$, which we decompose as a concatenation of oriented edges $e_1 \cup \cdots \cup e_n$, one says that $\gamma'$ is obtained from $\gamma$ by
\begin{itemize}
	\item \emph{flipping a square}, if there exists some $1 \leq i \leq n-1$ such that $$\gamma'= e_1 \cdots e_{i-1} \cup a \cup b \cup e_{i+2} \cup \cdots \cup e_n,$$ where $e_i,e_{i+1},b,a$ define an unoriented 4-cycle in $X$;
	\item \emph{shortening a triangle}, if there exists some $1 \leq i \leq n-1$ such that $$\gamma'= e_1 \cdots e_{i-1} \cup a \cup e_{i+2} \cup \cdots \cup e_n,$$ where $e_i,e_{i+1},a$ define an unoriented 3-cycle in $X$;
	\item \emph{removing a backtrack}, if there exists some $1 \leq i \leq n-1$ such that $$\gamma'= e_1 \cup \cdots \cup e_{i-1} \cup e_{i+2} \cup \cdots \cup e_n,$$ where $e_{i+1}$ is the inverse of $e_i$.
\end{itemize}
Lemma \ref{lem:GeodFlipSquare} follows from the simple connectivity of the prism complex associated to $X$ \cite[Theorem 2.120]{Qm} and from the fact that flipping squares and shortening triangles provide the relations of the fundamental groupoid of $X$; see \cite[Statement 9.1.6]{BrownGroupoidTopology} for more details.

\paragraph{Median graphs.} A graph $X$ is a \emph{median graph} if, for all vertices $x,y,z \in X$, there exists a unique vertex $m \in X$ such that
$$\left\{ \begin{array}{l} d(x,y)=d(x,m)+d(m,y) \\ d(x,z)= d(x,m)+d(m,z) \\ d(y,z) = d(y,m)+d(m,z) \end{array} \right..$$
The point $m$ is referred to as the \emph{median point} of the triple $x,y,z$. Median graphs are known to define the same objects as CAT(0) cube complexes. Indeed, the one-sketeton of a CAT(0) cube complex is a median graph; and the cube-completion of a median graph, namely the cube complex obtained by filling in all the one-skeleta of cubes in the graph with cubes, is a CAT(0) cube complex. We refer to \cite{mediangraphs} for more information. 

%
%
%

\paragraph{Hyperplanes.} Similarly to CAT(0) cube complexes, the notion of \emph{hyperplane} is fundamental in the study of quasi-median graphs.

\begin{definition}
Let $X$ be a graph. A \emph{hyperplane} $J$ is an equivalence class of edges with respect to the transitive closure of the relation saying that two edges are equivalent whenever they belong to a common triangle or are opposite sides of a square. We denote by $X \backslash \backslash J$ the graph obtained from $X$ by removing the interiors of all the edges of $J$. A connected component of $X \backslash \backslash J$ is a \emph{sector}. The \emph{carrier} of $J$, denoted by $N(J)$, is the subgraph generated by all the edges of $J$. Two hyperplanes $J_1$ and $J_2$ are \emph{transverse} if there exist two edges $e_1 \subset J_1$ and $e_2 \subset J_2$ spanning a square in $X$; and they are \emph{tangent} if they are not transverse but $N(J_1) \cap N(J_2) \neq \emptyset$. 
\end{definition}
\begin{figure}
\begin{center}
\includegraphics[trim={0 16.5cm 10cm 0},clip,scale=0.45]{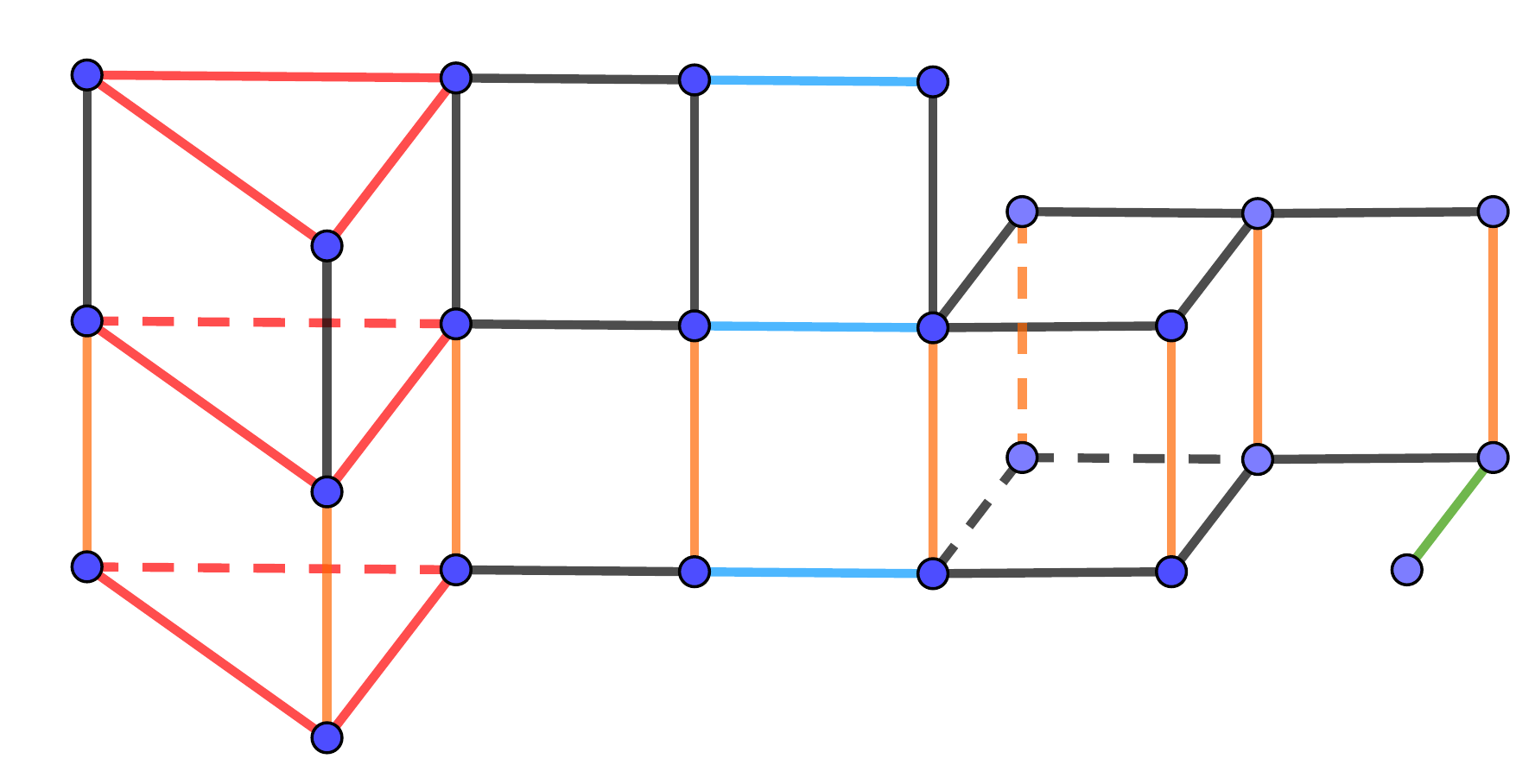}
\caption{A quasi-median graph and some of its hyperplanes.}
\label{figure3}
\end{center}
\end{figure}

\noindent
See Figure \ref{figure3} for examples of hyperplanes in a quasi-median graph. A key observation is that hyperplanes in quasi-median graphs always delimit at least two sectors:

\begin{thm}\label{thm:BigThmQM}\emph{\cite[Proposition 2.15]{Qm}}
Let $X$ be a quasi-median graph and $J$ a hyperplane. The graph $X \backslash \backslash J$ is disconnected, and the carrier and the sectors of $J$ are~gated. 
\end{thm}

\noindent
We refer to \cite[Section 2.2]{Qm} (and more particularly to \cite[Proposition 2.30]{Qm}) for more information about the (fundamental) connection between the geometry of quasi-median graphs and their hyperplanes.

\medskip \noindent
We record the following lemmas for future use.

\begin{lemma}\label{lem:DisjointCliques}\emph{\cite[Lemma 2.25]{Qm}}
In a quasi-median graph, two distinct cliques which are dual to the same hyperplane must be disjoint.
\end{lemma}

\begin{lemma}\label{lem:SpanSquare}\emph{\cite[Fact~2.70]{Qm}}
Let $X$ be a quasi-median graph and $e_1,e_2 \subset X$ two edges sharing their initial point. If the hyperplanes dual to $e_1$ and $e_2$ are transverse, then $e_1$ and $e_2$ span a square. 
\end{lemma}

\begin{lemma}\label{lem:GeodHypOrder}
Let $X$ be a quasi-median graph, $x,y \in X$ two vertices and $[x,y]$ a geodesic from $x$ to $y$. Let $J_1, \ldots, J_n$ denote the hyperplanes crossed by $[x,y]$ in that order. If $J_i$ and $J_{i+1}$ are transverse for some $1 \leq i \leq n-1$, then there exists a geodesic from $x$ to $y$ crossing the hyperplanes $J_1, \ldots, J_{i-1}, J_{i+1}, J_i, J_{i+2}, \ldots, J_n$ in that order. 
\end{lemma}

\begin{proof}
Decompose the geodesic $[x,y]$ as a concatenation of edges $e_1 \cup \cdots \cup e_n$. So, for every $1 \leq j \leq n$, $e_j$ is dual to the hyperplane $J_j$. As a consequence of Lemma~\ref{lem:SpanSquare}, the edges $e_i$ and $e_{i+1}$ span a square. Flipping this square (i.e., replacing $e_i$ and $e_{i+1}$ with their opposite edges in our square) produces a new path between $x$ and $y$ which has the same length as $[x,y]$ (and so is a geodesic) and which crosses the hyperplanes $J_1 ,\ldots, J_{i-1}, J_{i+1}, J_i, J_{i+2}, \ldots, J_n$ in that order.
\end{proof}

\paragraph{Graph products.} We conclude our preliminary section by considering graph products of groups and their quasi-median graphs.

\medskip \noindent
Let $\Gamma$ be a simplicial graph and $\mathcal{G}= \{ G_u \mid u \in V(\Gamma) \}$ be a collection of groups indexed by the vertex-set $V(\Gamma)$ of $\Gamma$. The \emph{graph product} $\Gamma \mathcal{G}$ is defined as the quotient
$$\left( \underset{u \in V(\Gamma)}{\ast} G_u \right) / \langle \langle [g,h]=1, g \in G_u, h \in G_v \ \text{if} \ \{u,v\} \in E(\Gamma) \rangle \rangle$$
where $E(\Gamma)$ denotes the edge-set of $\Gamma$. The groups in $\mathcal{G}$ are referred to as \emph{vertex-groups}. 

\medskip \noindent
\textbf{Convention.} In all the article, we will assume for convenience that the groups in $\mathcal{G}$ are non-trivial. Notice that it is not a restrictive assumption, since a graph product with some trivial factors can be described as a graph product over a smaller graph all of whose factors are non-trivial.

\medskip \noindent
A \emph{word} in $\Gamma \mathcal{G}$ is a product $g_1 \cdots g_n$ where $n \geq 0$ and where, for every $1 \leq i \leq n$, $g_i \in G$ for some $G \in \mathcal{G}$; the $g_i$ are the \emph{syllables} of the word, and $n$ is the \emph{length} of the word. Clearly, the following operations on a word does not modify the element of $\Gamma \mathcal{G}$ it represents:
\begin{description}
	\item[Cancellation:] delete the syllable $g_i=1$;
	\item[Amalgamation:] if $g_i,g_{i+1} \in G$ for some $G \in \mathcal{G}$, replace the two syllables $g_i$ and $g_{i+1}$ by the single syllable $g_ig_{i+1} \in G$;
	\item[Shuffling:] if $g_i$ and $g_{i+1}$ belong to two adjacent vertex-groups, switch them.
\end{description}
A word is \emph{graphically reduced} if its length cannot be shortened by applying these elementary moves. Every element of $\Gamma \mathcal{G}$ can be represented by a graphically reduced word, and this word is unique up to the shuffling operation. This allows us to define the \emph{length} of an element $g \in \Gamma \mathcal{G}$, denoted by $|g|$, as the length of any graphically reduced word representing $g$. For more information on graphically reduced words, we refer to \cite{GreenGP} (see also \cite{HsuWise,VanKampenGP}). 

\medskip \noindent
We record the following definition for future use:

\begin{definition}\label{def:headtail}
Given an element $g \in \Gamma \mathcal{G}$, the \emph{tail} of $g$ is the set of syllables of a graphically reduced word representing $g$ which appear as the last syllable in some new graphically reduced word representing $g$ obtained by shufflings.
\end{definition}

\noindent
The connection between graph products and quasi-median graphs is made explicit by the following statement \cite[Proposition 8.2]{Qm}:

\begin{thm}\label{thm:GPquasimedian}
Let $\Gamma$ be a simplicial graph and $\mathcal{G}$ a collection of groups indexed by $V(\Gamma)$. The Cayley graph $\QM := \mathrm{Cay} \left( \Gamma \mathcal{G}, \bigcup\limits_{G \in \mathcal{G}} G \backslash \{1 \} \right)$ is a quasi-median graph. 
\end{thm}

\noindent
Notice that the graph product $\Gamma \mathcal{G}$ naturally acts by isometries on $\QM$ by left-multiplication. We refer to \cite[Section 8.1]{Qm} for more information about the geometry of $\QM$. Here, we only mention the following two statements, which describe the cliques and the prisms of $\QM$:

\begin{lemma}\label{lem:GPcliques}\emph{\cite[Lemma 8.6]{Qm}}
Let $\Gamma$ be a simplicial graph and $\mathcal{G}$ a collection of groups indexed by $V(\Gamma)$. The cliques of $\QM$ are the subgraphs generated by the cosets of the form $g G_u$, where $g \in \Gamma \mathcal{G}$ and $u \in V(\Gamma)$.
\end{lemma}

\begin{lemma}\label{lem:GPprisms}\emph{\cite[Corollary 8.7]{Qm}}
Let $\Gamma$ be a simplicial graph and $\mathcal{G}$ a collection of groups indexed by $V(\Gamma)$. The prisms of $\QM$ are the subgraphs generated by the cosets of the form $g \langle \Lambda \rangle$, where $g \in \Gamma \mathcal{G}$, where $\Lambda$ is a complete subgraph of $\Gamma$, and where $\langle \Lambda \rangle$ denotes the subgroup of $\Gamma \mathcal{G}$ generated by the vertex-groups labelling the vertices of~$\Lambda$.
\end{lemma}

\section{Special actions on quasi-median graphs}

\subsection{Warm up: special cube complexes revisited}\label{section:warmup}

\noindent
As introduced in \cite{MR2377497}, special cube complexes are nonpositively curved cube complexes which do not contain configurations of hyperplanes referred to as \emph{pathological}. Then, the key observation is that, given such a cube complex $X$, there exists a graph $\Gamma$ (namely, the crossing graph of the hyperplanes in $X$) and a local isometry $X \hookrightarrow X_\Gamma$, where $X_\Gamma$ is a nonpositively curved cube complex with the right-angled Artin group $A_\Gamma$ as fundamental group (namely, a \emph{Salvetti complex}). As local isometries between nonpositively curved cube complexes are $\pi_1$-injective, it follows that $\pi_1(X)$ is a subgroup of $A_\Gamma$. Similar arguments can be conducted when $A_\Gamma$ is replaced with the right-angled Coxeter group $C_\Gamma$, but considering either $A_\Gamma$ or $C_\Gamma$ is essentially equivalent because a right-angled Artin groups always embeds as a finite-index subgroup into a right-angled Coxeter group \cite{MR1783167}.

\medskip \noindent
In this section, we sketch an alternative approach which illustrates the more general arguments from the next section. So we fix a group $G$ which acts \emph{specially} on a CAT(0) cube complex $X$, i.e.,
\begin{itemize}
	\item for every hyperplane $J$ and every element $g \in G$, $J$ and $gJ$ are neither transverse nor tangent;
	\item for all hyperplanes $J_1,J_2$ and every element $g \in G$, if $J_1$ and $J_2$ are transverse then $J_1$ and $gJ_2$ cannot be tangent.
\end{itemize}
Let $\Gamma$ denote the graph whose vertices are the $G$-orbits of hyperplanes and whose edges link two orbits if they contain at least two transverse hyperplanes. Naturally, each hyperplane of $X$ is labelled by a vertex of $\Gamma$, namely the $G$-orbits it belongs to. Define the \emph{label} of an oriented path $\gamma$ in $X$ as the word $\ell(\gamma)$ of the $G$-orbits of hyperplanes it crosses. Fixing a basepoint $x_0 \in X$, we consider
$$\Phi : \left\{ \begin{array}{ccc} X & \to & X(\Gamma) \\ x & \mapsto & \ell(\text{path from $x_0$ to $x$}) \end{array} \right.$$
where $X(\Gamma)$ denotes the usual CAT(0) cube complex on which the right-angled Coxeter group $C_\Gamma$ acts, namely the cube-completion of the Cayley graph $\mathrm{Cay}(C_\Gamma, V(\Gamma))$. Notice that $\Phi$ naturally induces
$$\varphi : \left\{ \begin{array}{ccc} G & \to & C_\Gamma \\ g & \mapsto & \Phi(g \cdot x_0) \end{array} \right..$$
It turns out that $\varphi$ is an injective morphism, that $\Phi$ is a $\varphi$-equivariant embedding, and that the image of $\Phi$ is a convex subcomplex of $C_\Gamma$. These observations are based on the following three claims, for which we sketch justifications.

\begin{claim}\label{claim:AnyPath}
The map $\Phi$ is well-defined, i.e., for every vertex $x \in X$, the label of a path from $x_0$ to $x$ does not depend on the path we choose.
\end{claim}

\noindent
First, consider an oriented path of the form $ee^{-1}$, namely a backtrack. Then $\ell(ee^{-1})= \ell(e)^2$ equals $1$ in $C_\Gamma$. Next, consider an oriented path of the form $e \cup f$ where $e$ and $f$ are consecutive edges in a square. Because the hyperplanes dual to $e$ and $f$ are transverse, the generators $\ell(e)$ and $\ell(f)$ commute in $C_\Gamma$, so 
$$\ell(e \cup f ) = \ell(e) \ell(f) = \ell(f) \ell(e)= \ell(e' \cup f')$$
where $e' \cup f'$ denotes the image of $e \cup f$ under the diagonal reflection in the square which contains $e \cup f$. Therefore, the label of a path remains the same if we add or remove a backtrack or if we \emph{flip a square}. In a CAT(0) cube complex, any two paths with the same endpoints can be obtained from one to another thanks to such elementary operations, so the desired conclusion follows. \qed

\begin{claim}
The map $\varphi : G \to C_\Gamma$ is a morphism.
\end{claim}

\noindent
Fix two elements $g,h \in G$. We have
$$\varphi(gh) = \ell([x_0,ghx_0])= \ell ( [x_0,gx_0] \cup g [x_0,hx_0]) = \ell([x_0,gx_0]) \ell( [x_0,hx_0]) = \varphi(g) \varphi(h),$$
where the second equality is justified by Claim \ref{claim:AnyPath}, and the third one by the fact that the labelling map $\ell$ is $G$-invariant. \qed

\medskip \noindent
So far, the specialness of the action has not been used, the morphism $\varphi : G \to C_\Gamma$ is well-defined for every action of $G$ on a CAT(0) cube complex. However, this assumption is crucial in the proof of the injectivity of $\Phi$ (and $\varphi$), which follows from the next assertion:

\begin{claim}
For every vertex $x \in X$ and every geodesic $[x_0,x]$ in $X$, the word $\ell([x_0,x])$ is graphically reduced in $C_\Gamma$.
\end{claim}

\noindent
Assume that there exists a vertex $x \in X$ and a geodesic $[x_0,x]$ such that the word $\ell([x_0,x])$ is not graphically reduced. So, if we write $[x_0,x]$ as a concatenation of oriented edges $e_1 \cup \cdots \cup e_n$, then there exist two indices $1 \leq i < j \leq n$ such that $\ell(e_i)= \ell(e_j)$ and such that $\ell(e_k)$ commutes with $\ell(e_i)$ for every $i < k < j$. Assume that $j -i \geq 2$. Because $\ell(e_i)$ and $\ell(e_{i+1})$ commute, the hyperplane $J_i$ dual to $e_i$ has a $G$-translate which is transverse to the hyperplane $J_{i+1}$ dual to $e_{i+1}$. Because the action is special, the hyperplanes $J_i$ and $J_{i+1}$ cannot be tangent, so they are transverse. As a consequence, the edges $e_i$ and $e_{i+1}$ span a square, and by flipping this square, we can replace our geodesic $[x_0,x]$ with a new geodesic so that $j-i$ decreases. By iterating the process, we end up with a geodesic $[x_0,x]$ such that $j=i+1$. In other words, $[x_0,x]$ contains two successive edges with the same label; or equivalently, if $J$ and $H$ denote the hyperplanes dual to these two edges, $J$ and $H$ belong to the same $G$-orbit. But $J$ and $H$ are either tangent or transverse, which contradicts the specialness of the action. \qed

\subsection{Embeddings into graph products}\label{section:embedding}

\noindent
In this section, we define special actions on quasi-median graphs and we show, given a group admitting such an action, how to embed it into a graph product. We begin by introducing the following notation:

\begin{notation}
Let $G$ be a group acting on a quasi-median graph $X$. For every hyperplane $J$ of $X$, $\mathscr{S}(J)$ denote the collection of sectors delimited by $J$, and $\mathfrak{S}(J)$ the image of $\mathrm{stab}_G(J)$ in the permutation group of $\mathscr{S}(J)$.
\end{notation}

\noindent
Special actions on quasi-median graphs are defined as follows:

\begin{definition}
Let $G$ be a group acting faithfully on a quasi-median graph $X$. The action is \emph{hyperplane-special} if
\begin{itemize}
	\item for every hyperplane $J$ and every element $g \in G$, $J$ and $gJ$ are neither transverse nor tangent;
	\item for all hyperplanes $J_1,J_2$ and every element $g \in G$, if $J_1$ and $J_2$ are transverse then $J_1$ and $gJ_2$ cannot be tangent.
\end{itemize}
The action is \emph{special} if, in addition, the action $\mathfrak{S}(J) \curvearrowright \mathscr{S}(J)$ is~free for every hyperplane~$J$ of $X$. 
\end{definition}

\noindent
It is worth noticing that our definition agrees with the definition of special actions on median graphs we used in the previous section. In other words, an action on a median graph is special if and only if it is hyperplane special. Indeed, hyperplanes in median graphs delimit exactly two sectors, and a faithful action on a set of cardinality two is automatically free. 

\medskip \noindent
The rest of the section is almost entirely dedicated to the proof of the following embedding theorem:

\begin{thm}\label{thm:BigEmbedding}
Let $G$ be a group acting specially on a quasi-median graph $X$. 
\begin{itemize}
	\item Fix representatives $\{J_i \mid i \in I\}$ of hyperplanes of $X$ modulo the action of $G$. 
	\item Let $\Gamma$ denote the graph whose vertex-set is $\{J_i \mid i\in I\}$ and whose edges link two hyperplanes if they have two transverse $G$-translates. 
	\item For every $i \in I$, let $G_i$ denote the group $\mathfrak{S}(J_i) \oplus K_i$, where $K_i$ is an arbitrary group of cardinality the number of orbits of $\mathfrak{S}(J_i) \curvearrowright \mathscr{S}(J_i)$. 
\end{itemize}
Then there exists an injective morphism $\varphi : G \hookrightarrow \Gamma \mathcal{G}$, where $\mathcal{G}= \{ G_i \mid i \in I\}$, and a $\varphi$-equivariant embedding $X \hookrightarrow \QM$ whose image is gated.
\end{thm}

\begin{proof}
First, notice that, for every $i \in I$, the free action $\mathfrak{S}(J_i) \curvearrowright \mathscr{S}(J_i)$ extends to a free and transitive action $G_i \curvearrowright \mathscr{S}(J_i)$. Indeed, by definition of $K_i$ there exists a bijection between $K_i$ and the set of $\mathfrak{S}(J_i)$-orbits of $\mathscr{S}(J_i)$, and $\mathfrak{S}(J_i)$ acts freely on each $\mathfrak{S}(J_i)$-orbit of $\mathscr{S}(J_i)$. Consequently, if for every $k \in K_i$ we fix a basepoint $x_k \in \mathscr{S}(J_i)$ in the $\mathfrak{S}(J_i)$-orbit labelled by $k$ under the previous bijection, then 
$$\left\{ \begin{array}{ccc} \mathfrak{S}(J_i) \oplus K_i & \to & \mathscr{S}(J_i) \\ (g,k) & \mapsto & g \cdot x_k \end{array} \right.$$
is a $\mathfrak{S}(J_i)$-equivariant bijection. Therefore, the action of $\mathfrak{S}(J_i) \oplus K_i$ on itself by left-multiplication can be transferred to a free and transitive action of $G_i$ on $\mathscr{S}(J_i)$ which extends the action of $\mathfrak{S}(J_i)$.

\medskip \noindent
From now on, we fix such extensions $G_i \curvearrowright \mathscr{S}(J_i)$, $i\in I$. As a consequence, the sectors delimited by $J_i$ can be naturally labelled by $G_i$. We want to extend such a labelling equivariantly to all the hyperplanes of $X$. 

\medskip \noindent
For every hyperplane $J$, let $S(J)$ denote the sector delimited by $J$ which contains a fixed basepoint $x_0 \in X$. 

\begin{claim}\label{claim:translate}
For every hyperplane $J$, there exist $i \in I$ and $g \in G$ such that $gJ = J_i$ and such that $g S(J)$ and $S(J_i)$ belong to the same $K_i$-orbit. 
\end{claim}

\noindent
Of course, there exist $i \in I$ and $h \in G$ such that $hJ=J_i$. Because $\mathfrak{S}(J_i)$ acts transitively on the $K_i$-orbits of $\mathscr{S}(J_i)$, there exists some $k \in \mathrm{stab}(J_i)$ such that $k \cdot hS(J)$ and $S(J_i)$ belong to the same $K_i$-orbit. Setting $g:=kh$ proves the claim. 

\medskip \noindent
\textbf{Labelling the sectors.} If $J$ is an arbitrary hyperplane of $X$, let $i \in I$ and $g \in G$ be as given by Claim \ref{claim:translate}. A sector $S$ delimited $J$ is labelled by the unique element of $G_i$ which sends $S(J_i)$ to $gS$. Let $\ell(S)$ denote this label. 

\medskip \noindent
Notice that the label of $S$ does not depend on the choice of $g$. Indeed, let $h \in G$ be another element satisfying Claim \ref{claim:translate}. Then $gh^{-1}$ stabilises $J_i$ and the sectors $gS(J)$, $S(J_i)$, $hS(J)$ all belong to the same $K_i$-orbit. In other words, $gh^{-1}$ defines an element of $\mathfrak{S}(J_i)$ which stabilises a $K_i$-orbit, which implies that $gh^{-1}$ represents the trivial element of $\mathfrak{S}(J_i)$, and we conclude that $g S(J)=gh^{-1} \cdot hS(J)=hS(J)$. 

\medskip \noindent
\textbf{Labelling the oriented paths.} If $e \subset X$ is an oriented edge, let $S_1$ (resp. $S_2$) denote the sector delimited by the hyperplane dual to $e$ which contains the initial endpoint of $e$ (resp. the terminal endpoint of $e$). The label of $e$ is defined as $\ell(e):= \ell(S_1)^{-1} \ell(S_2)$. More generally, if $\gamma= e_1 \cup \cdots \cup e_n$ is an oriented path, then its label is defined as the word $\ell(\gamma):= \ell(e_1) \cdots \ell(e_n)$, most of the time thought of as an element of $\Gamma \mathcal{G}$. 

\medskip \noindent
Because we may consider the label of an oriented path either as a word or as an element of $\Gamma \mathcal{G}$, we will use the following notation in order to avoid any ambiguity. Given two labels $a$ and $b$, we denote by $a=b$ the equality in the group $\Gamma \mathcal{G}$, and $a \equiv b$ the equality as words. 

\medskip \noindent
We record below two fundamental facts about the labelling we have constructed: it is $G$-invariant, and it sends geodesics to graphically reduced words.

\begin{claim}\label{claim:LabelInvariant}
Let $e \subset X$ be an oriented edge and $g \in G$ an element. Then $\ell(g \cdot e)=\ell(e)$. 
\end{claim}

\noindent
Let $J$ denote the hyperplane dual to $e$. According to Claim \ref{claim:translate}, there exist $i \in I$ and $h,k \in G$ such that $hJ=J_i=k \cdot gJ$ and such that $S(J_i)$, $hS(J)$, $k S(gJ)$ all belong to the same $K_i$-orbit of $\mathscr{S}(J_i)$. As a consequence, $kgh^{-1}$ stabilises $J_i$ so it defines an element $\sigma$ of $\mathfrak{S}(J_i)$. Notice that, if $S$ is an arbitrary sector delimited by $J$, then $\sigma$ sends $hS$ to $k \cdot gS$ (as elements of $\mathscr{S}(J_i)$). Because $\ell(S)$ is the unique element of $G_i$ sending $S(J_i)$ to $hS$ and that $\ell(gS)$ is the unique element of $G_i$ sending $S(J_i)$ to $kgS$, necessarily $\ell(gS)= \sigma \ell(S)$ in $G_i$. The key observation is that $\sigma$ does not depend on $S$. Therefore, if $S_1$ (resp. $S_2$) denotes the sector delimited by $J$ which contains the initial endpoint of $e$ (resp. the terminal endpoint of $e$), then 
$$\ell(ge)= \ell(gS_1)^{-1} \ell(gS_2)= \ell(S_1)^{-1} \sigma^{-1} \sigma \ell(S_2) = \ell(S_1)^{-1} \ell(S_2)= \ell(e),$$
concluding the proof of our claim.

\begin{claim}\label{claim:Embedding}
For all vertices $x,y \in X$ and every geodesic $[x,y]$ from $x$ to $y$, the word $\ell([x,y])$ is graphically reduced in $\Gamma \mathcal{G}$.
\end{claim}

\noindent
Assume for contradiction that there exist vertices $x,y \in X$ and a geodesic $[x,y]$ from $x$ to $y$, which we decompose as a concatenation of edges $e_1 \cup \cdots \cup e_r$, such that $\ell([x,y])$ is not graphically reduced in $\Gamma \mathcal{G}$. So there exist two indices $1 \leq i < j \leq r$ such that $\ell(e_i)$ and $\ell(e_j)$ belong to the same vertex-group of $\Gamma \mathcal{G}$ and such that $\ell(e_k)$ belongs to a vertex-group adjacent to the previous one for every $i < k<j$. In other words, if $J_k$ denotes the hyperplane dual to $e_k$ for every $1 \leq k \leq r$, then $J_i$ and $J_j$ belong to the same $G$-orbit and, for every $i<k<j$, a $G$-translate of $J_k$ is transverse to $J_j$. Because $G$ acts specially on $X$, notice that, if $j\geq i+2$, then the hyperplane $J_{j-1}$ cannot be tangent to $J_j$, so $J_{j-1}$ and $J_j$ are transverse. As a consequence of Lemma \ref{lem:GeodHypOrder}, there exists a geodesic from $x$ to $y$ which crosses the hyperplanes $J_1, \ldots, J_{j-2},J_j,J_{j-1},J_{j+1}, \ldots, J_r$ in that order. By iterating the argument, it follows that we can choose carefully our geodesic $[x,y]$ so that $j=i+1$. In other words, $J_i$ and $J_j$ are tangent or transverse. But we know that $J_i$ and $J_j$ belong to the same $G$-orbit, contradicting the specialness of the action. The proof of our claim is complete.

\medskip \noindent
\textbf{The embedding.} Fix a second basepoint $x_1 \in X$, possibly different from $x_0$. In order to prove our theorem, we want to show that
$$\Phi : \left\{ \begin{array}{ccc} X & \to & X(\Gamma, \mathcal{G}) \\ x & \mapsto & \ell \left( \text{path from $x_1$ to $x$} \right)  \end{array} \right.$$
defines an embedding whose image is gated, that
$$\varphi : \left\{ \begin{array}{ccc} G & \to & \Gamma \mathcal{G}  \\ g & \to & \Phi(g \cdot x_1) \end{array} \right.$$
is an injective morphism, and that $\Phi$ is $\varphi$-equivariant. 

\medskip \noindent
First of all, we claim that $\Phi$ is well-defined, i.e., the label of a path from $x_1$ to $x$ (as an element of $\Gamma \mathcal{G}$) does not depend on the path we choose. As a consequence of Lemma~\ref{lem:GeodFlipSquare}, it suffices to show that flipping a square, shortening a triangle and removing a backtrack do not modify the label of a path.

\medskip \noindent
We begin by noticing that, if $e \cup f$ is an oriented path between two opposite vertices of a square and if $e' \cup f'$ denotes the image of $e \cup f$ under the reflection along the diagonal of our square, then $e \cup f$ and $e' \cup f'$ have the same label. Indeed, observe that the endpoints of $e$ and $f'$ belong to the same sectors delimited by the hyperplane dual to $e$ and $f'$, and similarly for $f$ and $e'$, so $\ell(e \cup f) \equiv \ell(e) \ell(f)$ and $\ell(e' \cup f') \equiv \ell(f) \ell(e)$. But $\ell(e)$ and $\ell(f)$ belong to two vertex-groups of $\Gamma \mathcal{G}$ which are adjacent since the two hyperplanes dual to $e$ and $f$ are transverse. Therefore,
$$\ell(e \cup f ) =\ell(e) \ell(f)= \ell(f) \ell(e) = \ell(e' \cup f'),$$
so that flipping a square in a path does not modify its label (in $\Gamma \mathcal{G}$). Next, if $e \cup f$ is a backtrack, then
$$\ell(e \cup f) = \ell(e) \ell(f)= \ell(e) \ell(e)^{-1}= 1,$$
so that removing a backtrack to a path does not modify its label (in $\Gamma \mathcal{G}$) either. Finally, let $e \cup f$ be the concatenation of two successive edges in a triangle and let $e'$ denote the edge of this triangle with the same endpoints as $e\cup f$. Let $J$ denote the hyperplane containing our triangle, $S_1$ the sector delimited by $J$ which contains the initial point of $e$, $S_2$ the sector delimited by $J$ which contains the terminal endpoint of $e$, and $S_3$ the sector delimited by $J$ which contains the final point of $f$. Then
$$\ell(e \cup f ) = \ell(e) \ell(f)= \ell(S_1)^{-1} \ell(S_2) \cdot \ell(S_2)^{-1} \ell(S_3)= \ell(S_1)^{-1} \ell(S_3) = \ell(e'),$$
so that shortening a triangle does not modify the label of a path. Thus, we have proved that $\Phi$ is well-defined.

\medskip \noindent
It is worth noticing that our map $\Phi$ essentially does not depend on the basepoint $x_1$ we choose. When we allow the basepoint $x_1$ to vary, we denote by $\Phi_z$ the map obtained from $\Phi$ by replacing $x_1$ with another vertex $z \in X$. Then:

\begin{claim}\label{claim:Diagram}
For all vertices $p,q \in X$, we have the commutative diagram

\centerline{\xymatrix{
X \ar[r]^{\Phi_q} \ar[d]_{\Phi_p} & \QM \ar[dl]^{m_g} \\ \QM
}}

\noindent
where the isometry $m_g$ denotes the left-multiplication by $g:= \ell([p,q])$. 
\end{claim}

\noindent
Indeed, 
$$\Phi_p(x)= \ell([p,x]) = \ell([p,q] \cup [q,x]) = \ell([p,q]) \cdot \Phi_q(x)$$
for every vertex $x \in X$.

\medskip \noindent
We are now ready to show that $\varphi$ is an injective morphism and that $\Phi$ is a $\varphi$-equivariant embedding.

\begin{claim}\label{claim:injective}
The map $\Phi$ is an isometric embedding. In particular, it is injective.
\end{claim}

\noindent
Let $x,y \in X$ be two vertices. Fix a geodesic $[x,y]$ between $x$ and $y$ in $X$. As a consequence of Claim \ref{claim:Diagram}, 
$$d(\Phi(x), \Phi(y))= d( \Phi_x(x), \Phi_x(y)) = d(1, \ell([x,y]).$$
But $\ell([x,y]$ is a graphically reduced word according to Claim \ref{claim:Embedding}, so $d(1, \ell([x,y]))$ coincides with the length of $\ell([x,y])$, or equivalently with the number of edges of $[x,y]$. We conclude that $d(\Phi(x),\Phi(y))= d(x,y)$. 


\begin{claim}\label{claim:Equi}
For every $x \in X$ and $g \in G$, $\Phi(gx)= \varphi(g) \Phi(x)$.
\end{claim}

\noindent
By fixing arbitrary paths $[x_1,gx]$, $[x_1,gx_1]$ and $[x_1,x]$ in $X$, we have
$$\Phi(gx) = \ell([x_1,gx])= \ell([x_1,gx_1] \cup g[x_1,x]) = \ell([x_1,gx_1]) \ell([x_1,x])= \varphi(x) \Phi(x),$$
where the penultimate equality is justified by Claim \ref{claim:LabelInvariant}. Our claim is proved.

\medskip \noindent
Notice that Claim \ref{claim:Equi} implies that $\varphi$ is a morphism. Indeed, for every $g,h \in G$, we have
$$\varphi(gh) = \Phi(gh \cdot x_1)= \varphi(g) \Phi(h \cdot x_1)= \varphi(g) \varphi(h).$$
Moreover, the injectivity of $\varphi$ follows from the injectivity of $\Phi$ provided by Claim \ref{claim:injective}, and Claim \ref{claim:Equi} precisely means that $\Phi$ is $\varphi$-equivariant. It remains to show that the image of $\Phi$ is a gated subgraph of $\QM$. Notice that, thanks to Claim \ref{claim:injective}, we already know that it is an induced subgraph.

\begin{claim}\label{claim:CliqueLabel}
Let $x \in X$ be a vertex and $i \in I$ an index. If there exists some $a \in G_i$ such that $x$ is the initial vertex of an edge of $X$ labelled by $a$, then, for every $b \in G_i$, $x$ is the initial vertex of an edge labelled by $b$.
\end{claim}

\noindent
Fix an element $b \in G_i$. Let $C$ denote the clique of $X$ containing our edge labelled by $a$, and let $J$ denote the hyperplane of $X$ which contains it. By construction, the sectors delimited by $J$ are labelled by elements of $G_i$, and conversely every element of $G_i$ labels a sector delimited by $J$. Let $e$ be the edge of $C$ which connects $x$ to the sector delimited by $J$ which is labelled by $cb$, where $c \in G_i$ is the label of the sector containing $x$. Then $\ell(e)= c^{-1} \cdot cb = b$, so $e$ is the edge we are looking for.

\begin{claim}\label{claim:ImageTriangles}
The image under $\Phi$ of a clique of $X$ is a clique of $\QM$. As a consequence, the image of $\Phi$ contains its triangles.
\end{claim}

\noindent
Let $C$ be a clique of $X$. Fix an arbitrary vertex $x \in C$. The edges of $C$ are all labelled by the same group $G_i$, $i \in I$. It follows from Claim \ref{claim:CliqueLabel} that $\Phi(C) \supset \Phi(x) G_i$. On the other hand, $\Phi(x)G_i$ is a clique in $\QM$ according to Lemma \ref{lem:GPcliques}, so $\Phi(C) \subset \Phi(x) G_i$. The desired conclusion follows. 

\begin{claim}\label{claim:ImageLocConvex}
The image of $\Phi$ is locally convex. 
\end{claim}

\noindent
Let $e_1,e_2 \subset X$ be two edges which share their initial point and such that $\Phi(e_1)$ and $\Phi(e_2)$ span a square $S$. Necessarily, $\ell(e_1)$ and $\ell(e_2)$ belong to adjacent vertex-groups, which means that the hyperplane dual to $e_1$ has a $G$-translate which is transverse to the hyperplane dual to $e_2$. Because $G$ acts specially on $X$, it follows that the hyperplanes dual to $e_1$ and $e_2$ are transverse, so that $e_1$ and $e_2$ span a square in $X$ according to Lemma \ref{lem:SpanSquare}. The image of this square under $\Phi$ must be $S$ as $\QM$ does not contain $K_{3,2}$ as an induced subgraph, concluding the proof of our claim.

\medskip \noindent
By combining Lemma \ref{lem:GatedSub} with Claims \ref{claim:ImageTriangles} and \ref{claim:ImageLocConvex}, we conclude that the image of $\Phi$ is a gated subgraph. The proof of our theorem is complete. 
\end{proof}

\begin{remark}\label{remark:BiggerGroups}
It is worth noticing that, if replace $K_i$ in the statement of Theorem~\ref{thm:BigEmbedding} with an arbitrary group of cardinality \emph{at least} the number of orbits of $\mathfrak{S}(J_i) \curvearrowright \mathscr{S}(J_i)$, then our proof still shows that $G$ embeds into the (bigger) graph product $\Gamma \mathcal{G}$. Indeed, the precise assumption made on the cardinality of $K_i$ is only used in the proof of Claim~\ref{claim:CliqueLabel}. However, under this weaker assumption, the image of $X \hookrightarrow \QM$ is not longer gated, but only convex. 
\end{remark}

\begin{remark}
Observe that, if $X$ is a median graph in Theorem \ref{thm:BigEmbedding}, then the graph product we obtain is a right-angled Coxeter group. Indeed, for every $i \in I$, $\mathscr{S}(J_i)$ has cardinality two, so either $\mathfrak{S}(J_i)$ has order two and $K_i$ is trivial or $\mathfrak{S}(J_i)$ is trivial and $K_i$ has order two. Consequently, $G_i$ is cyclic of order two for every $i \in I$, and $\Gamma \mathcal{G}$ is a right-angled Coxeter group. So we recover that groups acting specially on a CAT(0) cube complexes embed into right-angled Coxeter groups.
\end{remark}

\noindent
When applying Theorem \ref{thm:BigEmbedding}, it may be difficult to understand the groups $\mathfrak{S}(J)$. Our next statement shows that, when the group acts with finitely many orbits of vertices, these groups are essentially clique-stabilisers (which are much easier to understand). 

\begin{prop}\label{prop:SJclique}
Let $G$ be a group acting specially on a quasi-median graph $X$ with finitely many orbits of vertices. For every hyperplane $J$ and every clique $C \subset J$, the image of $\mathrm{stab}(C)$ in $\mathfrak{S}(J)$ is faithful and has finite index.
\end{prop}

\noindent
We begin by proving a preliminary lemma.

\begin{lemma}\label{lem:TrivialVertexStab}
If $G$ acts specially on a quasi-median graph $X$, then vertex-stabilisers are trivial.
\end{lemma}

\begin{proof}
Assume that $g \in G$ fixes a vertex $x \in X$. 

\medskip \noindent
Let $y \in X$ be a neighbor of $x$. Let $C$ denote the clique which contains the edge connecting $x$ and $y$, and $J$ the hyperplane containing $C$. Because $J$ and $gJ$ are neither tangent nor transverse, necessarily $gJ=J$, so that $gC=C$ as a consequence of Lemma \ref{lem:DisjointCliques}. Because the action $\mathfrak{S}(J) \curvearrowright \mathscr{S}(J)$ is free, necessarily $g$ stabilises all the sectors delimited by $J$, which implies that $g$ fixes $C$ pointwise, and in particular $gy=y$. 

\medskip \noindent
Thus, we have proved that $g$ fixes $x$ and all its neighbors. By reproducing the argument to the neighbors of $y$, and so on, we deduce that $g$ fixes $X$ pointwise. As the action of $G$ on $X$ is faithful, we conclude that $g$ must be trivial.
\end{proof}

\begin{proof}[Proof of Proposition \ref{prop:SJclique}.]
Fix a hyperplane $J$ of $X$ and a clique $C \subset J$. The fact that the image of $\mathrm{stab}(C)$ in $\mathfrak{S}(J)$ is faithful is a direct consequence of Lemma \ref{lem:TrivialVertexStab}. Because $\mathfrak{S}(J)$ acts freely on $\mathscr{S}(J)$, it suffices to show that (the image of) $\mathrm{stab}(C)$ acts on $\mathscr{S}(J)$ with finitely many orbits in order to deduce that (the image of) $\mathrm{stab}(C)$ has finite index in $\mathfrak{S}(J)$. In fact, we claim that $\mathrm{stab}(C)$ acts on $C$ with finitely many orbits of vertices, which is sufficient to conclude. 

\medskip \noindent
Notice that, if two vertices $x$ and $y$ of $C$ are in the same $G$-orbit, then they are in the same $\mathrm{stab}(C)$-orbit. Indeed, let $g \in G$ be such that $gx=y$. Then the cliques $C$ and $gC$ are either identical or tangent. But they cannot be tangent because the action is special, so $g \in \mathrm{stab}(C)$. As $G$ acts on $X$ with finitely many orbits of vertices, it follows~that:

\begin{fact}\label{fact:CliqueFinitelyOrbits}
$C$ contains only finitely many $\mathrm{stab}(C)$-orbits of vertices.
\end{fact}

\noindent
This last observation concludes the proof of our proposition.
\end{proof}

\noindent
As a consequence of Proposition \ref{prop:SJclique}, we better understand the vertex-groups of the graph product into which we embed our group in Theorem \ref{thm:BigEmbedding}, under the additional assumption that the action on the quasi-median graph has only finitely many orbits of vertices.

\begin{cor}\label{cor:VertexGroups}
Let $G$ be a group which acts specially on a quasi-median graph $X$ with finitely many orbits of vertices. Following the notation in Theorem \ref{thm:BigEmbedding}, for every $i \in I$, $K_i$ is finite and $\mathfrak{S}(J_i)$ contains a clique-stabiliser as a finite-index subgroup; in particular, $G_i$ is a finite extension of a clique-stabiliser.
\end{cor}

\begin{proof}
As a consequence of Proposition \ref{prop:SJclique}, it is sufficient to show that, for every hyperplane $J$ of $X$, $\mathfrak{S}(J)$ acts on $\mathscr{S}(J)$ with finitely many orbits. This observation is a direct consequence of Fact \ref{fact:CliqueFinitelyOrbits}.
\end{proof}

\noindent
By combining Theorem \ref{thm:BigEmbedding} with Corollary \ref{cor:VertexGroups}, one immediately gets:

\begin{cor}
Let $G$ be a group which acts specially on a quasi-median graph $X$ with finitely many orbits of vertices. Then $G$ embeds into a graph product of finite extensions of clique-stabilisers. 
\end{cor}

\subsection{Gated-cocompact subgroups are virtual retracts}

\noindent
We saw in the previous section that a group acting specially on a quasi-median graph can be embedded into a graph product. In the present section, our goal is to show, under the additional assumption that the group acts with only finitely many orbits of vertices, that the image of this embedding is a virtual retract. Our proof is based on the following concept:

\begin{definition}
Let $G$ be a group acting on a quasi-median graph $X$. A subgroup $H \leq G$ is \emph{gated-cocompact} if there exists a gated subgraph $Y \subset X$ on which $H$ acts with finitely many orbits of vertices. 
\end{definition}

\noindent
Unless stated otherwise, a gated-cocompact subgroup of a graph product $\Gamma \mathcal{G}$ always refers to the action of $\Gamma \mathcal{G}$ on $\QM$. The main result of this section is that such subgroups are virtual retracts:

\begin{thm}\label{thm:VirtualRetract}
Let $\Gamma$ be a simplicial graph and $\mathcal{G}$ a collection of groups indexed by $V(\Gamma)$. A gated-cocompact subgroup $H \leq \Gamma \mathcal{G}$ is a virtual retract.
\end{thm}

\noindent
Before turning to the proof of our theorem, we need to introduce a few definitions. So let $X$ be a quasi-median graph and $G$ a group acting on it.
\begin{itemize}
	\item The \emph{rotative-stabiliser} of a hyperplane $J$ is $\mathrm{stab}_\circlearrowright(J):= \bigcap \{ \mathrm{stab}(C) \mid \text{$C \subset J$ clique} \}.$
	\item Given a $G$-invariant collection of hyperplanes $\mathcal{J}$, the action $G \curvearrowright X$ is \emph{$\mathcal{J}$-rotative} if, for every $J \in \mathcal{J}$, the action $\mathrm{stab}_\circlearrowright(J) \curvearrowright \mathscr{S}(J)$ is transitive and free.
	\item Given a vertex $x \in X$, a collection of hyperplanes $\mathcal{J}$ is \emph{$x$-peripheral} if there do not exist $J_1,J_2 \in \mathcal{J}$ such that $J_1$ separates $x$ from $J_2$.  
\end{itemize}
For instance, the action of $\Gamma \mathcal{G}$ on $\QM$ is \emph{fully rotative} \cite[Proposition 2.21]{AutGP}, i.e., it is $\mathcal{J}$-rotative where $\mathcal{J}$ denotes the collection of all the hyperplanes of $\QM$. 

\begin{lemma}\label{lem:PingPong}
Let $G$ be a group acting on a quasi-median graph $X$ with trivial vertex-stabilisers. Fix a basepoint $x_0 \in X$ and let $\mathcal{J}$ be an $x_0$-peripheral collection of hyperplanes. Assume that the action of $G$ on $X$ is $\mathcal{J}$-rotative. Then
$$Y:= \bigcap\limits_{J \in \mathcal{J}} \text{sector delimited by $J$ containing $x_0$}$$
is a fundamental domain for the action of $R:= \langle \mathrm{stab}_\circlearrowright (J) \mid J \in \mathcal{J} \rangle$ on $X$.
\end{lemma}

\begin{proof}
Let $x \in X$ be an arbitrary vertex. Assume that $x \notin Y$ and let $y \in Y$ denote its projection onto $Y$. The last edge of a geodesic $[x,y]$ must be dual to a hyperplane $J$ in $\mathcal{J}$. Because the action is $\mathcal{J}$-rotative, there exists some $g \in \mathrm{stab}_\circlearrowright(J)$ which sends $x$ in the sector delimited by $J$ which contains $Y$. Notice that $g$ sends $[x,y]$ minus its last edge to a path between $gx$ and $y$, so
$$d(gx,Y) \leq d(gx,y) \leq d(x,y)-1.$$
By iterating the argument, we conclude that there exists $r \in R$ such that $rx \in Y$.

\medskip \noindent
Now, fix an arbitrary vertex $x \in Y$. For every $J \in \mathcal{J}$, let $X_J$ denote the union of all the sectors delimited by $J$ which are disjoint from $Y$. Notice that:
\begin{itemize}
	\item If $J_1,J_2 \in \mathcal{J}$ are transverse, then $g_1$ and $g_2$ commute for every $g_1 \in \mathrm{stab}_\circlearrowright(J_1)$ and $g_2 \in \mathrm{stab}_\circlearrowright(J_2)$ \cite[Lemma 8.46]{Qm}.
	\item If $J_1,J_2 \in \mathcal{J}$ are transverse, then $g \cdot X_{J_2} \subset X_{J_2}$ for every $g \in \mathrm{stab}_\circlearrowright(J_1)$ \cite[Lemma 8.47]{Qm}.
	\item If $J_1,J_2 \in \mathcal{J}$ are distinct and not transverse, then $g \cdot X_{J_2} \subset X_{J_1}$ for every $g \in \mathrm{stab}_\circlearrowright(J_1) \backslash \{1\}$.
	\item for every $J \in \mathcal{J}$ and every $g \in \mathrm{stab}_\circlearrowright(J) \backslash \{1\}$, we have $g \cdot x \in X_J$.
\end{itemize}
Therefore, \cite[Proposition 8.44]{Qm} applies and we deduce from \cite[Fact 8.45]{Qm} that $g \cdot x \in \bigcup\limits_{J \in \mathcal{J}} X_J$ for every non-trivial $g \in R$; in particular, $g \cdot x \notin Y$. Thus, we have proved that $Y$ is a fundamental domain for $R \curvearrowright X$. 
\end{proof}

\begin{proof}[Proof of Theorem \ref{thm:VirtualRetract}.]
Let $Y \subset \QM$ be a gated subgraph on which $H$ acts with finitely many orbits of vertices. Let $\mathcal{J}$ denote the collection of the hyperplanes of $\QM$ which are tangent to $Y$. We set $R:= \langle \mathrm{stab}_\circlearrowright(J) \mid  J \in \mathcal{J} \rangle$ and $H^+:= \langle R,H \rangle$. Notice that $\mathcal{J}$ is $H$-invariant, so $R$ is a normal subgroup of $H^+$. Moreover, $Y$ coincides with 
$$\bigcap\limits_{J \in \mathcal{J}} \text{sector delimited by $J$ containing $Y$},$$
which is a fundamental domain of $R$ according to Lemma \ref{lem:PingPong}. Therefore, $H \cap R= \{1\}$. It follows that $H^+ = R \rtimes H$, so that $H$ is a retract in $H^+$. Moreover, since $Y$ is a fundamental domain of $R$ and because $H$ acts on $Y$ with finitely many orbits of vertices, necessarily $H^+$ acts on $\QM$ with finitely many orbits of vertices, which means that $H^+$ is a finite-index subgroup of $\Gamma \mathcal{G}$. Thus, we have proved that $H$ is a virtual retract in $\Gamma \mathcal{G}$. 
\end{proof}

\noindent
According to Theorem \ref{thm:BigEmbedding}, if a group $G$ acts specially on a quasi-median graph $X$ then there exists an embedding $\varphi : G \hookrightarrow \Gamma \mathcal{G}$ such that $X$ embeds $\varphi$-equivariantly into $\QM$ as a gated subgraph. As a consequence, if $G$ acts on $X$ with finitely many vertices, then the image of $\varphi$ is a gated-cocompact subgroup of $\Gamma \mathcal{G}$, so Theorem \ref{thm:VirtualRetract} directly implies that:

\begin{cor}
Let $G$ be a group which acts specially on a quasi-median graph $X$ with finitely many orbits of vertices. The image of the embedding $G \hookrightarrow \Gamma \mathcal{G}$ provided by Theorem \ref{thm:BigEmbedding} is a virtual retract in $\Gamma \mathcal{G}$.
\end{cor}

\noindent
It also follows from Theorem \ref{thm:VirtualRetract} that gated-cocompact subgroups of our group $G$ are virtual retracts in $G$ itself.

\begin{cor}\label{cor:GatedRetract}
Let $G$ be a group which acts specially on a quasi-median graph $X$ with finitely many orbits of vertices. Gated-cocompact subgroups of $G$ are virtual retracts in~$G$. 
\end{cor}

\begin{proof}
According to Theorem \ref{thm:BigEmbedding}, there exist a graph product $\Gamma \mathcal{G}$, an injective morphism $\varphi : G \hookrightarrow \Gamma \mathcal{G}$, and a $\varphi$-equivariant embedding $X \hookrightarrow \QM$ whose image is gated. As a consequence, any gated-cocompact subgroup $H$ of $G$ (with respect to its action on $X$) is a gated-cocompact subgroup of $\Gamma \mathcal{G}$ (with respect to its action on $\QM$). Therefore, $H$ is a virtual retract in $\Gamma \mathcal{G}$ according to Theorem \ref{thm:VirtualRetract}, which implies that $H$ is a virtual retract in $G$. 
\end{proof}

\noindent
As subgraphs in median graphs are gated if and only if they are convex, we recover from Corollary \ref{cor:GatedRetract} that convex-cocompact subgroups in cocompact special groups are virtual retracts \cite{MR2377497}.

\section{Right-angled graphs of groups}

\subsection{Graphs of groups}\label{section:IntroRAGG}

\noindent
We begin this section by fixing the basic definitions and notations related to graphs of groups; essentially, we follow \cite{MR1954121}. So far, our graphs were always one-dimensional simplicial complexes, but we need a different definition in order to define graphs of groups. In order to avoid ambiguity, we will refer to the latters as \emph{abstract graphs}. 

\begin{definition}
An \emph{abstract graph} is the data of a set of vertices $V$, a set of arrows $E$, a fixed-point-free involution $e \mapsto \bar{e}$ on $E$, and two maps $s,t : E \to E$ satisfying $t(e)= s \left( \bar{e} \right)$ for every $e \in E$. 
\end{definition}

\noindent
Notice that the elements of $E$ are referred to as arrows and not as edges. This terminology will allow us to avoid confusion between arrows of abstract graphs and edges of quasi-median graphs. Below, we define graphs of groups and their associated fundamental groupoids as introduced in \cite{Higgins}. 

\begin{definition}
A \emph{graph of groups}\index{Graphs of groups} $\mathfrak{G}$ is the data of an abstract graph $(V,E, \bar{\cdot}, s,t)$, a collection of groups indexed by $V \sqcup E$ such that $G_e= G_{\bar{e}}$ for every $e \in E$, and a monomorphism $\iota_e : G_e \hookrightarrow G_{s(e)}$ for every $e \in E$. The \emph{fundamental groupoid} $\mathfrak{F} = \mathfrak{F}(\mathfrak{G})$ of $\mathfrak{G}$ is the groupoid which has vertex-set $V$, which is generated by the arrows of $E$ together with $\bigsqcup\limits_{v \in V} G_v$, and which satisfies the relations:
\begin{itemize}
	\item for every $v \in V$ and $g,h,k \in G_v$, $gh=k$ if the equality holds in $G_v$;
	\item for every $e \in E$ and $g \in G_e$, $\iota_e(g) \cdot e = e \cdot \iota_{\bar{e}}(g)$. 
\end{itemize}
Notice in particular that, for every $e \in E$, $\bar{e}$ is an inverse of $e$ in $\mathfrak{F}$. Fixing some vertex $v \in V$, the \emph{fundamental group} of $\mathfrak{G}$ (based at $v$) is the vertex-group $\mathfrak{F}_v$ of $\mathfrak{F}$, i.e., the loops of $\mathfrak{F}$ based at $v$. 
\end{definition}

\noindent
We record the following definition for future use:

\begin{definition}
The \emph{terminus} of an element $g$ of $\mathfrak{F}$ is the vertex of $V$ which corresponds to the terminal point of $g$ when thought of as an arrow of $\mathfrak{F}$.
\end{definition}

\noindent
The following normal form, proved in \cite{Higgins}, is central in the quasi-median geometry of right-angled graphs of groups.

\begin{prop}\label{prop:GFnormalform}
Let $\mathfrak{G}$ be a graph of groups. For every $e \in E$, fix a left-transversal $T_e$ of $\iota_e \left( G_e \right)$ in $G_{s(e)}$ containing $1_{s(e)}$. Any element of $\mathfrak{F}$ can be written uniquely as a word $g_1 \cdot e_1 \cdots g_n \cdot e_n \cdot g_{n+1}$, where 
\begin{itemize}
	\item $(e_1,\ldots, e_n)$ is an oriented path in the underlying abstract graph;
	\item $g_i \in T_{e_i}$ for every $1 \leq i \leq n$, and $g_{n+1}$ is an arbitrary element of $G_{t(e_n)}$;
	\item if $e_{i+1} = \bar{e_i}$ for some $1 \leq i \leq n-1$ then $g_{i+1} \neq 1$.
\end{itemize}
\end{prop}

\noindent
Such a word will be referred to as a \emph{normal word}. 

\medskip \noindent
Roughly speaking, we will be interested in graphs of groups gluing graph products. In order to get something interesting for our purpose, we need to control the gluings.

\begin{definition}
Given two graph products $\Gamma \mathcal{G}$ and $\Lambda \mathcal{H}$, a morphism $\Phi : \Gamma \mathcal{G} \to \Lambda \mathcal{H}$ is a \emph{graphical embedding} is there exist an embedding $f : \Gamma \to \Lambda$ and isomorphisms $\varphi_v : G_v \to H_{f(v)}$, $v \in V(\Gamma)$, such that $f(\Gamma)$ is an induced subgraph of $\Lambda$ and $\Phi(g)= \varphi_v(g)$ for every $v \in V(\Gamma)$ and every $g \in G_v$. 
\end{definition}

\noindent
Typically, we glue graph products along ``subgraph products'' in a canonical way. We refer to Section \ref{section:RAGGex} for examples.

\begin{definition}
A \emph{right-angled graph of groups} is a graph of groups such that each (vertex- and edge-)group has a fixed decomposition as a graph product and such that each monomorphism of an edge-group into a vertex-group is a graphical embedding (with respect to the structures of graph products we fixed). 
\end{definition}

\noindent
Fix a right-angled graph of groups $\mathfrak{G}$. For every arrow $e \in E$, there exists a natural left-transversal $T_e$ of $\iota_e(G_e)$ in $G_{s(e)}$: the set of graphically reduced words of $G_{s(e)}$ whose tails (see Definition \ref{def:headtail}) do not contain any element of the vertex-groups in $\iota_e(G_e)$. From now on, we fix this choice, and any normal word will refer to this convention.

\medskip \noindent
In the following, a \emph{factor} $G$ will refer to a vertex-group of one of these graph products. In order to avoid possible confusion, in the sequel vertex-groups will only refer to the groups labelling the vertices of the underlying abstract graph of $\mathfrak{G}$.


\subsection{Quasi-median geometry}\label{section:QMforRAGG}

\noindent
Fix a right-angled graph of groups $\mathfrak{G}$, and a vertex $\omega \in V$ of its underlying abstract graph. Let $\mathfrak{S} \subset \mathfrak{F}$ denote the union of the arrows of $E$ together with the factors (minus the identity) of the graph products $G_v$, $v \in V$. By definition, $\mathfrak{S}$ is a generating set of the fundamental groupoid $\mathfrak{F}$ of $\mathfrak{G}$. 

\begin{definition}
The graph $\mathfrak{X}= \mathfrak{X}(\mathfrak{G},\omega)$ is the connected component of the Cayley graph $\mathfrak{X}(\mathfrak{G})$ of the groupoid $\mathfrak{F}$, constructed from the generating set $\mathfrak{S}$, which contains the neutral element $1_{\omega}$ based at $\omega$. In other words, $\mathfrak{X}$ is the graph whose vertices are the arrows of $\mathfrak{F}$ starting from $\omega$ and whose edges link two elements $g,h \in \mathfrak{F}$ if $g=h \cdot s$ for some $s \in \mathfrak{S}$. 
\end{definition}

\noindent
It is worth noticing that an edge of $\mathfrak{X}$ is naturally labelled either by an arrow of $E$ or by a factor. 

\begin{prop}\label{prop:RAGGqm}\emph{\cite[Theorem 11.8]{Qm}}
The graph $\mathfrak{X}$ is quasi-median.
\end{prop}

\noindent
Notice that the fundamental group $\mathfrak{F}_{\omega}$ of $\mathfrak{G}$ based at $\omega$ naturally acts by isometries on $\mathfrak{X}$ by left-multiplication. Moreover:

\begin{lemma}\label{lem:OrbitsInX}
Two vertices of $\mathfrak{X}$ belong to the same $\mathfrak{F}_\omega$-orbit if and only if they have the same terminus.
\end{lemma}

\begin{proof}
If $g \in \mathfrak{F}_\omega$ and $h \in \mathfrak{X}$, it is clear that $h$ and $gh$ have the same terminus. Conversely, if $h,k \in \mathfrak{X}$ have the same terminus, then the product $kh^{-1}$ is well-defined and it belongs to $\mathfrak{F}_\omega$. Since $kh^{-1} \cdot h =k$, it follows that $h$ and $k$ belong to the same $\mathfrak{F}_\omega$-orbit.
\end{proof}

\noindent
We record the following definition for future use:

\begin{definition}
A \emph{leaf} of $\mathfrak{X}$ is the subgraph generated by the set of vertices $gG_v$, where $G_v$ is a vertex-group of $\mathfrak{G}$ and where $g \in \mathfrak{F}$ is some arrow starting from $\omega$ and ending at $v \in V$.
\end{definition}

\noindent
Notice that, by construction, a leaf is isometric to the Cayley graph of a graph product as given by Theorem \ref{thm:GPquasimedian}. (See \cite[Lemma 11.11]{Qm} for more details.)

\paragraph{Path morphisms.} Let $\mathfrak{G}$ be a right-angled graph of groups and let $(V,E, \bar{\cdot},s,t)$ denote its underlying abstract graph. Given an arrow $e \in E$, we denote by $\varphi_e : \iota_e(G_e) \to \iota_{\bar{e}}(G_e)$ the isomorphism $\iota_{\bar{e}} \circ \iota_e^{-1}$. A priori, $\varphi_e$ is not defined on $G_{s(e)}$ entirely, but for every subset $S \subset G_{s(e)}$, we can define $\varphi_e(S)$ as $\varphi_e \left( S \cap \iota_e(G_e) \right)$. By extension, if an oriented path $\gamma$ decomposes as a concatenation of arrows $e_1 \cup \cdots \cup e_n$, we denote by $\varphi_\gamma$ the composition $\varphi_{e_n} \circ \cdots \circ \varphi_{e_1}$. 

\medskip \noindent
Notice that, if $G$ is a factor contained in a vertex-group $G_u$ of $\mathfrak{G}$ and if $\gamma$ is a path in the graph of $\mathfrak{G}$ starting from $u$, then $\varphi_\gamma(G)$ is either empty or a factor (different from $G$ in general). Moreover, in the latter case, the equality
$$a \cdot e_1 \cdots e_n = e_1 \cdots e_n \cdot \varphi_\gamma (a)$$
holds for every $a \in G$, where $e_1 \cup \cdots \cup e_n$ is a decomposition of $\gamma$ as a concatenation of arrows.

\medskip \noindent
Given a right-angled graph of groups, a subgroup of automorphisms is naturally associated to each factor:

\begin{definition}
For every factor $G$ contained in a vertex-group $G_u$ of $\mathfrak{G}$, 
$$\Phi(G):= \{ \varphi_c \mid \text{$c$ loop based at $u$ such that $\varphi_c(G)=G$} \} \leq \mathrm{Aut}(G).$$
\end{definition}

\noindent
These groups of automorphisms are crucial in the study of right-angled graphs of groups. Indeed, as noticed by \cite[Example 11.36]{Qm}, cyclic extensions of an arbitrary group are fundamental groups of right-angled graphs of groups, but we cannot expect to find a geometry common to all the cyclic extensions, so we need additional restrictions on the graphs of groups we look at. As suggested by \cite[Proposition 11.26]{Qm} and Proposition~\ref{prop:WhenSpecial}, typically we require the $\Phi(G)$ to be trivial, or at least finite.

\paragraph{Cliques and prisms.} Let $\mathfrak{G}$ be a right-angled graph of groups. The description of the cliques in our quasi-median graph $\mathfrak{X}$ is given by the following lemma. 

\begin{lemma}\label{lem:RAGGclique}\emph{\cite[Lemma 11.15]{Qm}}
A clique of $\mathfrak{X}$ is either by an edge labelled by an arrow or a complete subgraph $gG$ where $G$ is a factor and $g \in \mathfrak{F}$. 
\end{lemma}

\noindent
About the prisms of $\mathfrak{X}$, notice that we already understand the prisms which lie in leaves, as a consequence of Lemma \ref{lem:GPprisms}. The other prisms are described by our next lemma.

\begin{lemma}\label{lem:RAGGprism}\emph{\cite[Lemma 11.18]{Qm}}
For every prism $Q$ of $\mathfrak{X}$ which is not included in a leaf, there exist some $e \in E$ and some prism $P$ which is included into a leaf, such that $Q$ is generated by the set of vertices $\{ g, \ ge \mid g \in P \}$.
\end{lemma}

\paragraph{Hyperplanes.} Let $\mathfrak{G}$ be a right-angled graph of groups. The rest of the section is dedicated to the description of the hyperplanes of $\mathfrak{X}$. It is worth noticing that, as a consequence of \cite[Fact 11.14 and Lemma 11.16]{Qm}, a hyperplane has all its edges labelled either an arrow of $\mathfrak{G}$ or by factors (not a single one in general). In the former case, the hyperplane is \emph{of arrow-type}; and in the latter case, the hyperplane is \emph{of factor-type}. Notice that, as a consequence of \cite[Fact 11.14]{Qm}, two hyperplanes of arrow-type cannot be transverse.

\medskip \noindent
Roughly speaking, the carrier of the hyperplane dual to a clique labelled by some factor $G$ is generated by the vertices corresponding to elements of $\mathfrak{F}$ which ``commute'' with all the elements of $G$. Because commutation is not well-defined in groupoids, we need to define carefully this idea, which is done by the following definition.

\begin{definition}\label{def:LinkRAGG}
Let $G$ be a factor contained in a vertex-group $G_v$ of $\mathfrak{G}$. An element $h \in \mathfrak{F}$ belongs to the \emph{link} of $G$, denoted by $\mathrm{link}(G)$, if it can be written as a normal word $h_1e_1 \cdots h_ne_nh_{n+1}$ such that:
\begin{itemize}
	\item for every $1 \leq i \leq n-1$, $\varphi_{e_i} \left( \cdots \left( \varphi_{e_1} (G) \right) \cdots \right)$ is non-empty and included in $\iota_{e_{i+1}}(G_{e_{i+1}})$;
	\item for every $1 \leq i \leq n$, $h_i$ belongs to a factor adjacent the factor $\varphi_{e_{i-1}} \left( \cdots \left( \varphi_{e_{1}}(G) \right) \cdots \right)$ in the graph product $G_{s(e_i)}$;
	\item $h_{n+1}$ belongs to a factor adjacent to the factor $\varphi_{e_n} \left( \cdots \left( \varphi_{e_1}(G) \right) \cdots \right)$ in the graph product $G_{t(e_n)}$.
\end{itemize}
\end{definition}

\noindent
We are now ready to describe the hyperplanes of factor-type of $\mathfrak{X}$ and their stabilisers.

\begin{prop}\label{prop:RAGGhyp}\emph{\cite[Proposition 11.21]{Qm}}
Let $C=gG$ be a clique where $G$ is a factor and where $g \in \mathfrak{F}$. Let $J$ denote the hyperplane dual to $C$. An edge $e \subset \mathfrak{X}$ is dual to $J$ if and only if $e=g(h_1 \ell,h_2 \ell)$ for some $h_1,h_2 \in G$ distinct and $\ell \in \mathrm{link}(G)$. As a consequence, $N(J)= g G \cdot \mathrm{link}(G)$ and the fibers of $J$ are the $gh \cdot \mathrm{link}(G)$ where $h \in G$.
\end{prop}

\begin{cor}\label{cor:RAGGstabhyp}\emph{\cite[Corollary 11.22]{Qm}}
Let $C=gG$ be a clique where $G$ is a factor and where $g \in \mathfrak{F}$. Let $J$ denote the hyperplane dual to $C$. Then
$$\mathrm{stab}(J)=g \{ k h \mid k \in G, \ h \in \mathrm{link}(G), \varphi_h(G)=G \} g^{-1} .$$
\end{cor}

\noindent
In this statement, $\varphi_h$ is defined as follows. Writing $h$ as a normal word $h_1e_1 \dots h_ne_nh_{n+1}$ as in Definition \ref{def:LinkRAGG} (this representation being unique according to Proposition \ref{prop:GFnormalform}), we refer to $e_1 \cup \cdots \cup e_n$ as the \emph{path associated to $h$}. Then, $\varphi_h:= \varphi_{e_1 \cup \cdots \cup e_n}$. Notice that, by definition of $\mathrm{link}(G)$, $\varphi_h$ always sends $G$ to another factor; or, in other words, $\varphi_h(G)$ cannot be empty.

\medskip \noindent
About the hyperplanes of arrow-type of $\mathfrak{X}$, a complete description is not required here. The following statement will be sufficient:

\begin{lemma}\label{lem:RAGGhypE}\emph{\cite[Lemma 11.24]{Qm}}
Let $J$ be a hyperplane of arrow-type in $\mathfrak{X}$. Then $J$ has exactly two fibers, and they are both stabilised by $\mathrm{stab}(J)$. 
\end{lemma}

\subsection{When is the action special?}\label{section:WhenSpecial}

\noindent
In this section, we want to understand when the action of the fundamental group of a right-angled graph of groups on the quasi-median graph constructed in Section \ref{section:QMforRAGG} is special. Our main result in this direction is the following statement.

\begin{prop}\label{prop:WhenSpecial}
Let $\mathfrak{G}$ be a right-angled graph of groups. The action of the fundamental group $\mathfrak{F}_\omega$ of $\mathfrak{G}$ on $\mathfrak{X}(\mathfrak{G},\omega)$ is special if and only if the following conditions are satisfied: 
\begin{itemize}
	\item[(i)] for every factor $G$ and every cycle $c$ in the graph of $\mathfrak{G}$ based at the vertex-group containing $G$, $\varphi_c(G) = \emptyset$ or $G$;
	\item[(ii)] there does not exist two vertices $u,v$ in the graph of $\mathfrak{G}$, two paths $\alpha,\beta$ from $u$ to $v$, two commuting factors $A_1,A_2 \subset G_u$ and two non-commuting factors $B_1,B_2 \subset G_v$ such that $\varphi_\alpha(A_1)=B_1$ and $\varphi_\beta(A_2)= B_2$;
	\item[(iii)] in the graph of $\mathfrak{G}$, an edge have distinct endpoints;
	\item[(iv)] for every factor $G$, the equality $\Phi(G)= \{ \mathrm{Id} \}$ holds. 
\end{itemize}
\end{prop}

\noindent
We begin by proving the following preliminary lemma:

\begin{lemma}\label{lem:RACGsJ}
Let $\mathfrak{G}$ be a right-angled graph of groups. Let $G$ be a factor of $\mathfrak{G}$ and let $C$ denote a clique labelled by $G$, say $C=gG$. Also, let $J$ denote the hyperplane containing $C$. The action $\mathfrak{S}(J) \curvearrowright \mathscr{S}(J)$ is free and transitive if and only if $\Phi(G)= \{ \mathrm{Id}\}$. Moreover, if this is the case, then the image of $\mathrm{stab}(C) = gGg^{-1}$ in $\mathfrak{S}(J)$ is faithful and surjective.
\end{lemma}

\begin{proof}
As a consequence of Proposition \ref{prop:RAGGhyp}, it is clear that $\mathrm{stab}(C)=gGg^{-1}$ acts faithfully, freely and transitively on $\mathscr{S}(J)$. Therefore, the action $\mathfrak{S}(J) \curvearrowright \mathscr{S}(J)$ is free and transitive if and only if the image of $\mathrm{stab}(C) = gGg^{-1}$ in $\mathfrak{S}(J)$ is surjective. According to Corollary \ref{cor:RAGGstabhyp}, this amounts to saying that $g\{ m \in \mathrm{link}(G) \mid \varphi_m(G)=G \}g^{-1}$ acts trivially on $\mathscr{S}(J)$, or equivalently, as a consequence of Proposition \ref{prop:RAGGhyp}, that $\Phi(G)=\{ \mathrm{Id} \}$. 
\end{proof}

\noindent
The next observation will be fundamental in our proof:

\begin{lemma}\label{lem:TransferFactor}
Let $\mathfrak{G}$ be a right-angled graph of groups, and $e,f \subset \mathfrak{X}$ two edges. Let $A,B$ denote the two factors labelling $e,f$ respectively, and let $u,v$ denote the vertices of the graph of $\mathfrak{G}$ such that $A$ and $B$ are factors of $G_u$ and $G_v$ respectively. If $e$ and $f$ are dual to the same hyperplane, then there exists a path $\gamma$ in the graph of $\mathfrak{G}$ from $u$ to $v$ such that $\varphi_\gamma(A)=B$.
\end{lemma}

\begin{proof}
Write $e=(p,pa)$ and $f=(q,qb)$ where $a \in A$ and $b \in B$. As a consequence of Proposition \ref{prop:RAGGhyp}, $f=p(a_1 \ell, a_2 \ell)$ for some distinct $a_1,a_2 \in A$ and some $\ell \in \mathrm{link}(A)$. We have
$$b = q^{-1} \cdot qb = \ell^{-1} a_1^{-1}p^{-1} \cdot pa_2\ell = \varphi_{\ell} \left( a_1^{-1}a_2 \right) = \varphi_\gamma \left( a_1^{-1}a_2 \right)$$
where $\gamma$ is the path in the graph of $\mathfrak{G}$ associated to $\ell$. Because $\varphi_\gamma$ sends a factor to the empty set or to another factor, we conclude that $\gamma$ is a path from $u$ to $v$ and that $\varphi_\gamma(A)=B$, as desired. 
\end{proof}

\noindent
Now we are ready to determine when the action of the fundamental group of a right-angled graph of groups on its quasi-median graph is hyperplane-special.

\begin{lemma}\label{lem:WhenHypSpecial}
Let $\mathfrak{G}$ be a right-angled graph of groups. The action of the fundamental group $\mathfrak{F}_\omega$ of $\mathfrak{G}$ on $\mathfrak{X}(\mathfrak{G},\omega)$ is hyperplane-special if and only if the following conditions are satisfied: 
\begin{itemize}
	\item[(i)] for every factor $G$ and every cycle $c$ in the graph of $\mathfrak{G}$ based at the vertex-group containing $G$, $\varphi_c(G) = \emptyset$ or $G$;
	\item[(ii)] there does not exist two vertices $u,v$ in the graph of $\mathfrak{G}$, two paths $\alpha,\beta$ from $u$ to $v$, two commuting factors $A_1,A_2 \subset G_u$ and two non-commuting factors $B_1,B_2 \subset G_v$ such that $\varphi_\alpha(A_1)=B_1$ and $\varphi_\beta(A_2)= B_2$;
	\item[(iii)] in the graph of $\mathfrak{G}$, an edge have distinct endpoints.
\end{itemize}
\end{lemma}

\begin{proof}
First, assume that the action of the fundamental group of $\mathfrak{G}$ on $\mathfrak{X}$ is not hyperplane-special. There are several cases to consider.

\medskip \noindent
\textbf{Case 1:} There exist a hyperplane $J$ of $\mathfrak{X}$ and an element $g \in \mathfrak{F}_\omega$ such that $gJ$ and $J$ are transverse or tangent.

\medskip \noindent
It is clear that, if there exist two distinct intersecting edges which are labelled by the same arrow, then this arrow provides an edge of the graph of $\mathfrak{G}$ whose endpoints coincide, so that $(iii)$ does not hold. So, from now on, we assume that $J$ is of factor-type. Fix two distinct edges $e_1 \subset J$ and $e_2 \subset gJ$ which share their initial point, and let $A,B$ denote the distinct factors which label them. Notice that, because $e_1$ and $e_2$ intersect, our factors $A$ and $B$ belong to the same vertex-group of $\mathfrak{G}$, say $G_u$. Because $ge_1$ is labelled by the factor $A$ and is dual to the same hyperplane as $e_2$, namely $gJ$, it follows from Lemma \ref{lem:TransferFactor} that there exists in the graph of $\mathfrak{G}$ a loop $c$ based at $u$ such that $\varphi_c(A)=B$. In particular, $\varphi_c(A)$ is neither empty nor $A$, contradicting $(i)$. 

\medskip \noindent
\textbf{Case 2:} There exist two tangent hyperplanes $J_1,J_2$ of $\mathfrak{X}$ and an element $g\in \mathfrak{F}_\omega$ such that $J_1$ and $g J_2$ are transverse. 

\medskip \noindent
We distinguish three cases, depending on whether $J_1$ and $J_2$ are of arrow-type or of factor-type.

\medskip \noindent
\textbf{Case 2.1:} $J_2$ is of arrow-type.

\medskip \noindent
Fix a geodesic $\gamma \subset N(J_1)$ whose initial point belongs to $N(J_1) \cap N(J_2)$ and whose last edge is dual to $gJ_2$. Crossing $J_2$ corresponds to right-multiplying by the arrow $e$ (or its inverse) which labels $J_2$. But such a multiplication is allowed only if the element of the groupoid we are considering has its terminus which is an endpoint of $e$ (the initial or terminal point of $e$ depending on whether we are multiplying by $e$ or $e^{-1}$). Consequently, the initial point of $\gamma$ and one of the last two points of $\gamma$ have the same terminus. According to Lemma~\ref{lem:OrbitsInX}, these two points belong to the same $\mathfrak{F}_\omega$-orbit. So there exists some $h \in \mathfrak{F}_\omega$ such that the initial point of $\gamma$ belongs to $N(hgJ_2) \cap N(hJ_1)$. We already know from Case 1 that, if $J_1$ and $hJ_1$ are tangent or transverse, then $(i)$ cannot hold, so (since their carriers intersect) we suppose that they coincide. Similarly, we suppose that $J_2=hgJ_2$. As $gJ_2$ and $J_1$ are transverse, it follows that $hgJ_2$ and $hJ_1$ must be transverse as well; but $J_2$ and $J_1$ are tangent, a contradiction.

\medskip \noindent
\textbf{Case 2.2:} $J_1$ and $J_2$ are both of factor-type.

\medskip \noindent
Fix two edges $e_1 \subset J_1$ and $e_2 \subset J_2$ which share their initial point, and let $A_1$ and $A_2$ denote the factors which label them respectively. Notice that, because $e_1$ and $e_2$ intersect, $A_1$ and $A_2$ belong to the same vertex-group of $\mathfrak{G}$, say $G_u$. Moreover, because $e_1$ and $e_2$ do not span a square, $A_1$ and $A_2$ do not commute in the graph product $G_u$. Next, fix two edges $f_1 \subset J_1$ and $f_2 \subset gJ_2$ which share their initial endpoint and which span a square, and let $B_1$ and $B_2$ denote the factors which label them respectively. Notice that, because $f_1$ and $f_2$ intersect, $B_1$ and $B_2$ belong to the same vertex-group of $\mathfrak{G}$, say $G_v$. Moreover, because $f_1$ and $f_2$ span a square, $B_1$ and $B_2$ commute in the graph product $G_v$. As $e_1$ and $f_1$ are dual to the same hyperplane, namely $J_1$, it follows from Lemma \ref{lem:TransferFactor} that there exists a path $\alpha$ in the graph of $\mathfrak{G}$ from $u$ to $v$ such that $\varphi_\alpha(A_1)=B_1$. Similarly, because $f_2$ and $ge_2$ are dual to $gJ_2$, there exists a path $\beta$ from $u$ to $v$ such that $\varphi_\beta(A_2)=B_2$. We conclude that $(ii)$ does not hold.

\medskip \noindent
\textbf{Case 2.3:} $J_1$ is of arrow-type and $J_2$ of factor-type.

\medskip \noindent
Fix two edges $e_1 \subset J_1$ and $e_2 \subset J_2$ which share their initial point. Let $A$ denote the factor labelling $e_2$ and $e$ the arrow labelling $e_1$. Also, fix two edges $f_1 \subset J_1$ and $f_2 \subset gJ_2$ which share their initial point and which span a square. Notice that $f_1$ is labelled by $e$ or $\bar{e}$. We choose $f_1$ and $f_2$ such that their common initial point belongs to the same sector delimited by $J_1$ as the initial point of $e_1$ and $e_2$. As a consequence, $f_1$ is labelled by $e$, and the factor, say $B$, which labels $f_2$ belongs to the same vertex-group of $\mathfrak{G}$ as $A$, say $G_u$. As the edges $ge_2$ and $f_2$ are both dual to $gJ_2$, it follows from Lemma \ref{lem:TransferFactor} that there exists a loop $c$ in the graph of $\mathfrak{G}$, based at $u$, such that $\varphi_c(A)=B$. Notice that, because $e_1$ and $e_2$ do not span a square, necessarily $A$ does not belong to the image of $G_e$ in $G_u$. On the other hand, because $f_1$ and $f_2$ span a square, necessarily $B$ belongs to the image of $G_e$ in $G_u$. Consequently, $A \neq B$. We conclude that $\varphi_c(A)$ is neither empty nor $A$, contradicting $(i)$.

\medskip \noindent
Thus, we have proved that, if the conditions $(i)$, $(ii)$ and $(iii)$ of our proposition hold, then the action of $\mathfrak{F}_\omega$ on $\mathfrak{X}$ is hyperplane-special. Conversely, assume that one of the conditions $(i)$, $(ii)$ and $(iii)$ does not hold. 

\medskip \noindent
If $(i)$ does not hold, then there exist a loop $c$ in the graph of $\mathfrak{G}$ based at some vertex $u$ and a factor $G$ in the graph product $G_u$ such that $\varphi_c(G)$ is a factor of $G_u$ distinct from $G$. Fix an arbitrary vertex $h \in \mathfrak{X}$ whose terminus is $u$ (for instance, a concatenation of arrows from $\omega$ to $u$). Also, fix a non-trivial element $g \in G$ and write $c$ as a concatenation of arrows $e_1 \cup \cdots \cup e_n$. Notice that, for every $0 \leq i \leq n$, we have
$$hge_1 \cdots e_i = h e_1 \cdots e_i \varphi_{e_1 \cup \cdots \cup e_i}(g) \ \text{where} \ \varphi_{e_1 \cup \cdots \cup e_i} (g) \neq 1,$$
so $he_1 \cdots e_i$ and $hge_1 \cdots e_i$ are adjacent vertices. As a consequence, for every $0 \leq i \leq n-1$, the four vertices $he_1 \cdots e_i$, $hge_1 \cdots e_i$, $he_1 \cdots e_{i+1}$ and $hge_1 \cdots e_{i+1}$ span a square in $\mathfrak{X}$, so the two edges $(h,hg)$ and $(he_1 \cdots e_n, hge_1 \cdots e_n)$ are dual to the same hyperplane, say $J$. But $he_1 \cdots e_n h^{-1} \in \mathfrak{F}_\omega$ sends the edge $(h,gh)$ to the edge $(he_1 \cdots e_n, he_1 \cdots e_ng)$, and the two edges $(he_1 \cdots e_n, hge_1 \cdots e_n)$ and $(he_1 \cdots e_n, he_1 \cdots e_ng)$ are distinct because
$$hge_1 \cdots e_n = h e_1 \cdots e_n \varphi_c(g) \ \text{where $\varphi_c(g) \notin G$}.$$
Therefore, the hyperplanes $he_1 \cdots e_n h^{-1} J$ and $J$ are either tangent or transverse (depending on whether $G$ and $\varphi_c(G)$ commute in $G_u$). So the action of $\mathfrak{F}_\omega$ on $\mathfrak{X}$ is not hyperplane-special.

\medskip \noindent
If $(ii)$ does not hold, then there exist two vertices $u,v$ in the graph of $\mathfrak{G}$, a path $\alpha$ from $u$ to $v$, a path $\beta$ from $v$ to $u$, two commuting factors $A_1,A_2 \subset G_u$ and two non-commuting factors $B_1,B_2 \subset G_v$ such that $\varphi_\alpha(A_1)=B_1$ and $\varphi_\beta(B_2)=A_2$. Fix an arbitrary vertex $h$ of $\mathfrak{X}$ whose terminus is $v$ (for instance, a concatenation of arrows from $\omega$ to $v$) and non-trivial elements $p \in B_1$, $b \in B_2$, $a \in A_1$. Also, write $\alpha$ as the concatenation of arrows $a_1 \cdots a_s$ and $\beta$ as $b_1 \cdots b_r$. Notice that, for every $0 \leq i \leq r$, the vertices $hb_1 \cdots b_i$ and $hbb_1 \cdots b_i$ are adjacent as
$$hbb_1 \cdots b_i = h b_1 \cdots b_i \varphi_{b_1 \cup \cdots \cup b_i}(b) \ \text{where} \ \varphi_{b_1 \cup \cdots \cup b_i} (b) \neq 1.$$
Consequently, for every $0 \leq i \leq r-1$, the vertices $hb_1 \cdots b_i$, $hbb_1 \cdots b_i$, $hb_1 \cdots b_{i+1}$ and $hbb_1 \cdots b_{i+1}$ span a square. See Figure \ref{HypSpecial}. Similarly, for every $0 \leq i \leq s$, the vertices $h \beta a_1 \cdots a_i$ and $h \beta a a_1 \cdots a_i$ are adjacent because
$$h \beta a a_1 \cdots a_i = h \beta a_1 \cdots a_i \varphi_{a_1 \cup \cdots \cup a_i}(a) \ \text{where} \ \varphi_{a_1 \cup  \cdots \cup a_i}(a) \neq 1;$$
so, for every $0 \leq i \leq s-1$, the vertices $h\beta a_1 \cdots a_i$, $h \beta a_1 \cdots a_{i+1}$, $h \beta a a_1 \cdots a_i$ and $h \beta a a_1 \cdots a_{i+1}$ span a square. Notice that the edges $(h,hb)$ and $(h,hp)$ do not span a square because $B_2$ and $B_1$ do not commute, so the hyperplane $J_1$ dual to $(h,hb)$ is tangent to the hyperplane $J_2$ dual to $(h,hp)$. Next, because $A_1$ and $A_2$ commute, we have
$$hb \beta a = h \beta \varphi_\beta(b) a = h \beta a \varphi_\beta(b) \ \text{where} \ \varphi_\beta(b) \in A_2 \backslash \{ 1 \},$$
so the vertices $h \beta$, $hb \beta$, $h \beta a$ and $hb \beta a$ span a square. As a consequence, the hyperplane $J_3$ dual to the edge $(h \beta,h\beta a)$ is transverse to $J_1$. Finally, observe that $\beta \alpha$ is a loop based at $v$ in the graph of $\mathfrak{G}$, so $g:= h \beta \alpha h^{-1}$ represents an element of $\mathfrak{F}_\omega$. Moreover, $g(h,hp)= (h \beta \alpha, h \beta \alpha p)$ belongs to the same clique as the edge $(h \beta \alpha, h \beta a \alpha)$ because
$$h \beta a \alpha = h\beta \alpha \varphi_{\alpha}(a) \ \text{and} \ p, \varphi_{\alpha}(a) \in B_1,$$
hence $J_3=gJ_2$. Thus, we have proved that $J_1$ and $J_2$ are tangent but $J_1$ and $gJ_2$ are transverse, showing that the action of $\mathfrak{F}_\omega$ on $\mathfrak{X}$ is not hyperplane-special.
\begin{figure}
\begin{center}
\includegraphics[scale=0.35]{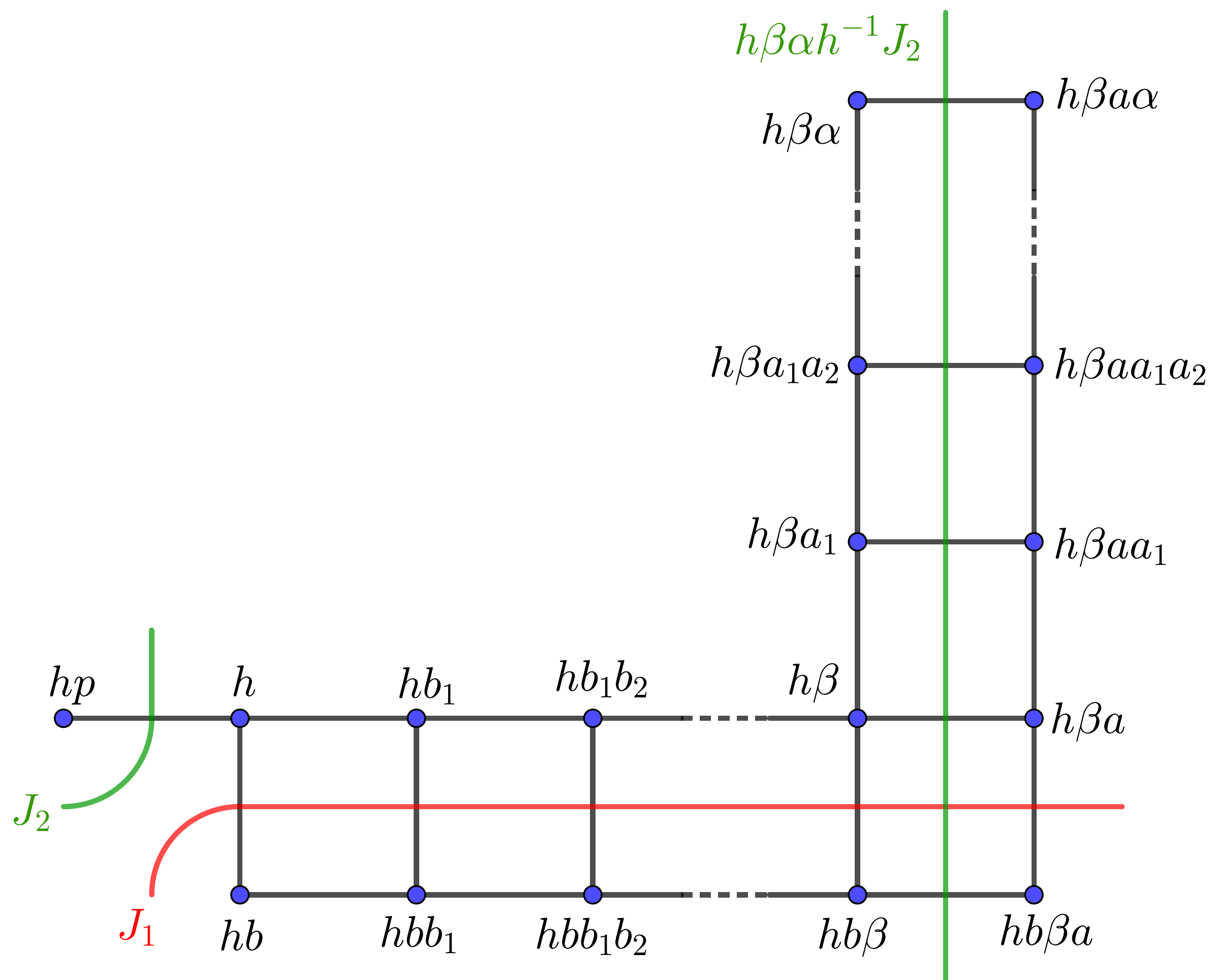}
\caption{Configuration of vertices when $(ii)$ does not hold.}
\label{HypSpecial}
\end{center}
\end{figure}

\medskip \noindent
Finally, if $(iii)$ does not hold, then there exists an arrow $e$ which is a loop based at some vertex $u$ of the graph of $\mathfrak{G}$. Fix an arbitrary vertex $h$ of $\mathfrak{X}$ whose terminus is $u$ (for instance a concatenation of arrows from $\omega$ to $u$). Then $h^{-1} e h$ defines an element of $\mathfrak{F}_\omega$ which acts on the bi-infinite line $\{ he^n \mid n \in \mathbb{Z} \} \subset \mathfrak{X}$ as a translation of length one. Consequently, if $J$ is any hyperplane crossing this line, then $J$ and $h^{-1}eh J$ are tangent, proving that the action of $\mathfrak{F}_\omega$ on $\mathfrak{X}$ is not hyperplane-special.
\end{proof}

\begin{proof}[Proof of Proposition \ref{prop:WhenSpecial}.]
Our proposition is an immediate consequence of Lemmas~\ref{lem:RACGsJ}, \ref{lem:RAGGhypE} and \ref{lem:WhenHypSpecial}.
\end{proof}

\noindent
As a consequence of Proposition \ref{prop:WhenSpecial}, one obtains a sufficient condition which implies that the fundamental group of a right-angled graph of groups embeds into a graph product. Our next proposition describes such a graph product; we refer to Figure \ref{Graph} for an illustration of the graph constructed in its statement. 
\begin{figure}
\begin{center}
\includegraphics[scale=0.32]{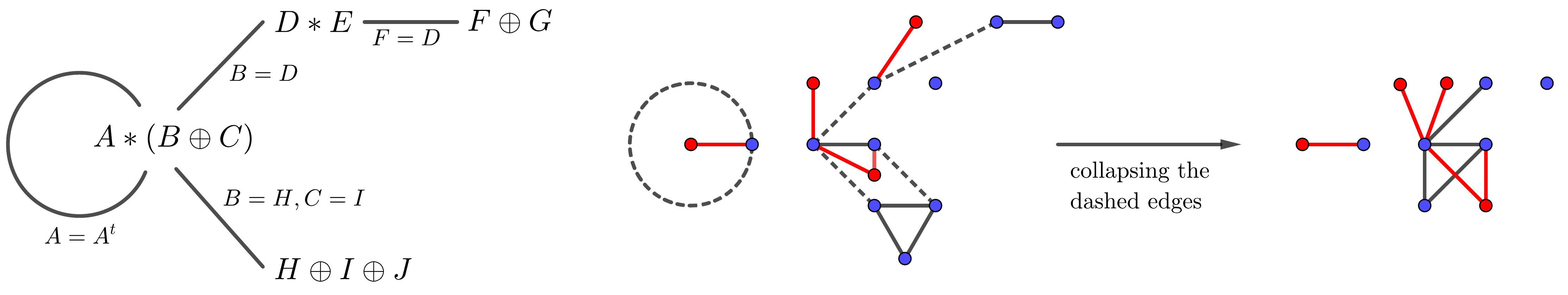}
\caption{A graph of groups and the graph $\Psi$ that Proposition \ref{prop:GraphForEmbedding} associates to it.}
\label{Graph}
\end{center}
\end{figure}

\begin{prop}\label{prop:GraphForEmbedding}
Let $\mathfrak{G}$ be a right-angled graph of groups satisfying the assumptions of Proposition \ref{prop:GraphForEmbedding}. Let $\Gamma$ denote the underlying graph of $\mathfrak{G}$, and, for every vertex $u \in V(\Gamma)$, let $\Gamma_u$ denote the graph corresponding to our decomposition of the vertex-group $G_u$ as a graph product. Given two vertices $a \in \Gamma_u$ and $b \in \Gamma_v$, write $a \sim b$ if $u$ and $v$ are linked by an arrow $e$ in $\Gamma$ and if $\varphi_e$ sends the factor corresponding to $a$ to the factor corresponding to $b$. Let $\Psi$ denote the graph obtained from $\Psi_0:= \left( \bigcup\limits_{u \in V(\Gamma)} \Gamma_u \right) / \sim$ by adding a vertex for each arrow $e \in E(\Gamma)$ and by linking $e$ to each vertex of $\Gamma_{s(e)}$ corresponding to a factor in the image of $\iota_e$. Finally, let $\mathcal{G}$ denote the collection of groups indexed by $V(\Psi)$ such that a group indexed by (the image in $\Psi_0$ of) a vertex of $\Gamma_u$ is the corresponding factor and such that the groups indexed by arrows are cyclic of order two. Then $\mathfrak{F}_\omega$ embeds into $\Psi \mathcal{G}$.
\end{prop}

\begin{proof}
As a consequence of Lemma \ref{lem:RAGGclique}, $\mathfrak{F}_\omega$-orbits of cliques in $\mathfrak{X}$ are bijectively indexed by the vertices of $E(\Gamma) \cup \bigcup\limits_{u \in V(\Gamma)} \Gamma_u$. Notice that two hyperplanes labelled by distinct arrows lie in distinct $\mathfrak{F}_\omega$-orbits of hyperplanes and that:

\begin{claim}\label{claim:PathFactor}
Let $A$ and $B$ be two factors respectively in the vertex-groups $G_u$ and $G_v$. Two hyperplanes $J$ and $H$ dual to cliques $C$ and $D$ respectively labelled by $A$ and $B$ belong to the same $\mathfrak{F}_\omega$-orbit if and only if there exists a path $\gamma$ from $u$ to $v$ such that $\varphi_\gamma(A)=B$.
\end{claim} 

\noindent
If $J$ and $H$ belong to the same $\mathfrak{F}_\omega$-orbit, then there exists some $g \in \mathfrak{F}_\omega$ such that $g C$ and $D$ are dual to the same hyperplane, namely $H$. The desired conclusion follows from Lemma \ref{lem:TransferFactor}. 

\medskip \noindent
Conversely, assume that there exists a path $\gamma$ from $u$ to $v$ such that $\varphi_\gamma(A)=B$. Write $C=gA$ and $D=hB$ for some $g$ and $h$, and write $\gamma$ as a concatenation of arrows $e_1 \cup \cdots \cup e_k$. Also fix a non-trivial element $a \in A$. Notice that, for every $0 \leq i \leq k$, the vertices $ge_1 \cdots e_i$ and $gae_1 \cdots e_i$ are adjacent because
$$gae_1 \cdots e_i= ge_1 \cdots e_i \varphi_{e_1 \cup \cdots \cup e_i}(a) \ \text{where} \ \varphi_{e_1 \cup \cdots \cup e_i}(a) \neq 1.$$
Consequently, for every $0 \leq i \leq k-1$, the vertices $ge_1 \cdots e_i$, $gae_1 \cdots e_i$, $ge_1 \cdots e_{i+1}$ and $gae_1 \cdots e_{i+1}$ span a square. It follows that the edges $(g,ga) \subset C$ and $(ge_1 \cdots e_k,gae_1 \cdots e_k)$ are dual to the same hyperplane, namely $J$. By noticing that
$$gae_1 \dots e_k = g e_1 \cdots e_k \varphi_{e_1 \cup \cdots \cup e_k}(a) \ \text{where} \ \varphi_{e_1 \cup \cdots \cup e_k}(a)= \varphi_\gamma(a) \in B,$$
we deduce that the edge $(ge_1 \cdots e_k, gae_1 \cdots e_k)$ is a translate an edge of the clique $D$. As a consequence, $J$ and $H$ belong to the same $\mathfrak{F}_\omega$-orbit, concluding the proof of our claim.

\medskip \noindent
So far, we have proved that the $\mathfrak{F}_\omega$-orbits of hyperplanes in $\mathfrak{X}$ are bijectively indexed by the vertices of $\Psi$. Next, notice that two hyperplanes of arrow-type cannot be transverse. Moreover:

\begin{claim}
Let $A$ and $B$ be two factors respectively in the vertex-groups $G_u$ and $G_v$. Two hyperplanes $J$ and $H$ dual to cliques $C$ and $D$ respectively labelled by $A$ and $B$ admit transverse $\mathfrak{F}_\omega$-translates if and only if there exists a path $\gamma$ from $u$ to $v$ such that $\varphi_\gamma(A)$ and $B$ are two distinct commuting factors.
\end{claim}

\noindent
If $J$ and $H$ admit transverse $\mathfrak{F}_\omega$-translates, then there exists an element $g \in \mathfrak{F}_\omega$ and a clique $E \subset \mathfrak{X}$ such that $E$ and $D$ span a prism and $gC$ and $E$ are dual to the same hyperplane. From the former assertion, we deduce that the factors labelling $E$ and $D$ belong to the same vertex-group of $\mathfrak{G}$, namely $G_v$, that they are distinct and that they commute; and from the latter assertion, as a consequence of Claim \ref{claim:PathFactor}, we know that there exists a path $\gamma$ from $u$ to $v$ such that $\varphi_\gamma(A)$ labels the clique $E$. Therefore, $\varphi_\gamma(A)$ commutes with $B$, as desired.

\medskip \noindent
Conversely, assume that there exists a path $\gamma$ from $u$ to $v$ such that $\varphi_\gamma(A)$ and $B$ are two distinct commuting factors. Clearly, there exists a clique $E \subset \mathfrak{X}$ labelled by $\varphi_\gamma(A)$ such that $E$ and $D$ span a prism. It follows from Claim \ref{claim:PathFactor} that the hyperplanes dual to the cliques $C$ and $E$ belong to the same orbit. Consequently, $J$ and $H$ admit transverse translates, concluding the proof of our claim.

\begin{claim}
Let $C$ be a clique of $\mathfrak{X}$ labelled by a factor $A$, say belonging to a vertex-group $G_u$ of $\mathfrak{G}$, and let $H$ be a hyperplane labelled by an arrow $e$. Let $J$ denote the hyperplane dual to $C$. The hyperplanes $J$ and $H$ admit transverse translates if and only if there exists a path $\gamma$ from $u$ to $s(e)$ such that $\varphi_\gamma(A)$ lies in the image of $\iota_e$. 
\end{claim}

\noindent
Assume that $J$ and $H$ admit transverse translates. Then there exist an element $g \in \mathfrak{F}_\omega$ and a clique $E \subset \mathfrak{X}$ such that $E$ spans a prism with an edge labelled by $e$ and such that $E$ and $gC$ are dual to the same hyperplane. From the former assertion, it follows that $E$ is labelled by a factor in $G_{s(e)}$ which is included in the image of $\iota_e$; and from the latter assertion, as a consequence of Claim \ref{claim:PathFactor}, we deduce that there exists a path $\gamma$ from $u$ to $s(e)$ such that $\varphi_\gamma(A)$ labels the clique $E$. The desired conclusion follows.

\medskip \noindent
Conversely, assume that there exists a path $\gamma$ from $u$ to $s(e)$ such that $\varphi_\gamma(A)$ lies in the image of $\iota_e$. Clearly, there exists a clique $E$ labelled by $\varphi_\gamma(A)$ which spans a prism with an edge labelled by $e$. As a consequence of Claim \ref{claim:PathFactor}, $E$ and $C$ belong to the same orbit. Moreover, any two edges labelled by the same arrow belong to the same orbit of hyperplanes. Consequently, $J$ and $H$ must admit transverse translates, concluding the proof of our claim.

\medskip \noindent
So far, we have proved that $\Psi$ coincides with the graph whose vertices are the orbits of hyperplanes and whose edges link two orbits if they contain transverse hyperplanes. Next, notice that, if $J$ is a hyperplane containing a clique labelled by a factor $G$, then $\mathfrak{S}(J)$ is isomorphic to $G$ and it acts transitively on $\mathscr{S}(J)$ according to Lemma \ref{lem:RACGsJ}. And if $J$ is a hyperplane labelled by an arrow, then, according to Lemma \ref{lem:RAGGhypE}, $\mathfrak{S}(J)$ is trivial and $\mathscr{S}(J)$ has cardinality two. Therefore, the embedding described by our proposition follows from Theorem \ref{thm:BigEmbedding}.
\end{proof}

\subsection{Examples}\label{section:RAGGex}

\noindent
In practice, Proposition \ref{prop:WhenSpecial} most of the time does not apply, its assumptions are just too strong. However, it turns out that the conditions $(i)-(iii)$ are often satisfied \emph{up to a finite cover}, so that the condition $(iv)$ seems to be the central condition of our criterion.

\medskip \noindent
Let $\mathfrak{G}$ be an arbitrary graph of groups and let $\Gamma$ denote its underlying graph. If $\pi : \Gamma' \to \Gamma$ is a cover, then we naturally defines a graph of groups $\mathfrak{G}'$ which has $\Gamma'$ as its underlying graph by defining, for every vertex $u \in V(\Gamma')$ and every edge $e \in E(\Gamma')$, the vertex-group $G_u$ as $G_{\pi(u)}$, the edge-group $G_e$ as $G_{\pi(e)}$ and the monomorphism $\iota_e : G_e \hookrightarrow G_{s(e)}$ as $\iota_{\pi(e)} : G_{\pi(e)} \hookrightarrow G_{s(\pi(e))}$. One obtains a covering of graphs of groups $\mathfrak{G}' \to \mathfrak{G}$ as defined in \cite{MR1239551}, so that the fundamental group of $\mathfrak{G}'$ embeds into the fundamental group of $\mathfrak{G}$; moreover, if $\Gamma' \to \Gamma$ is a finite cover, then the image of this embedding has finite index. (More topologically, one can say that the (finite sheeted) cover $\Gamma' \to \Gamma$ induces a (finite sheeted) cover from the graph of spaces defining $\mathfrak{G}'$ to the graph of spaces defining $\mathfrak{G}$; see \cite{MR564422} for more details on graphs of spaces and their connection with graphs of groups.)

\medskip \noindent
Although taking a well-chosen finite cover of graphs of groups often allows us to apply Proposition \ref{prop:WhenSpecial}, we were not able to prove that this strategy always work, and leave the following question open (for which we expect a positive answer).

\begin{question}
Let $\mathfrak{G}$ be a right-angled graph of groups. Assume that the graph of $\mathfrak{G}$ is finite, that its vertex-groups are graph products over finite graphs, and that $\Phi(G)$ is finite for every factor $G$. Does there exists a finite cover $\mathfrak{G}' \to \mathfrak{G}$ such that $\mathfrak{G}'$ satisfies the assumptions of Proposition \ref{prop:WhenSpecial}?
\end{question}

\noindent
In the rest of the section, we explain how to exploit Proposition \ref{prop:WhenSpecial} in specific examples. The examples of right-angled graphs of groups given below are taken from \cite{Qm}. We emphasize that, as a consequence of Remark \ref{remark:BiggerGroups}, in the embeddings given below, the $\mathbb{Z}_2$ can be replaced with arbitrary non-trivial groups.

\begin{ex}
Given a group $A$, consider the graph of groups with a single vertex, labelled by $A \times A$, and a single edge, labelled by $A$, such that the edge-group $A$ in sent into $A \times A$ first as the left-factor and next as the right-factor. Let $A^\rtimes$ denote the fundamental group of this graph of groups. The group $A^\rtimes$ admits $$\langle A,t \mid [a,tat^{-1}]=1, \ a \in A \rangle$$ as a (relative) presentation. Notice that, if $A$ is infinite cyclic, we recover the group introduced in \cite{MR897431}, which was the first example of fundamental group of a 3-manifold which is not subgroup separable. 

\medskip \noindent
By construction, $A^\rtimes$ is the fundamental group of a right-angled graph of groups, so it acts on a quasi-median graph. However, the conditions $(i)$ and $(iii)$ in Proposition \ref{prop:WhenSpecial} are not satisfied, so this action is not special. Nevertheless, it is sufficient to consider a new graph of groups, which is a $2$-sheeted cover of the previous one. 

\medskip \noindent
More generally, fix another group $B$, and consider the graph of groups which has two vertices, both labelled by $A \times B$, and two edges between these vertices, labelled by $A$ and $B$, such that the edge-group $A$ is sent into the vertex-groups as the left-factor $A$ and such that the edge-group $B$ is sent into the vertex-groups as the right-factor $B$. The fundamental group of this graph of groups is denoted by $A\square B$, and has
$$\langle A,B, t \mid [a,b]= [a,tbt^{-1}]=1, \ a \in A, b \in B \rangle$$
as a (relative) presentation. Observe that $A \square A$ is naturally a subgroup of $A^\rtimes$ of index two, and that the right-angled graph of groups defining $A \square B$ satisfies the assumptions of Proposition \ref{prop:WhenSpecial}. Let $\Gamma$ denote the graph which is a path of length three $a-b-c-d$ and let $\mathcal{G}_{A,B}= \{G_a=\mathbb{Z}_2, G_b=A, G_c=B, G_d= \mathbb{Z}_2 \}$. By applying Proposition~\ref{prop:GraphForEmbedding}, it follows that $A \square B$ embeds into $\Gamma \mathcal{G}_{A,B}$. Such an embedding is given by sending $A \subset A \square B$ to $A \subset \Gamma \mathcal{G}_{A,B}$, $B \subset A \square B$ to $B \subset \Gamma \mathcal{G}_{A,B}$ and $t \in A \square B$ to $xy \in \Gamma \mathcal{G}_{A,B}$ where $x \in G_a,y \in G_b$ are non-trivial. 

\medskip \noindent
Thus, we have found a subgroup $A \square A$ of index two in $A^\rtimes$ and we have constructed an embedding $A\square A \hookrightarrow \Gamma \mathcal{G}_{A,A}$ whose image is a virtual retract. 

\medskip \noindent
Notice that, if we replace the $\mathbb{Z}_2$ with infinite cyclic groups (as allowed by Remark~\ref{remark:BiggerGroups}), then it follows that the group $\mathbb{Z}^\rtimes$ from \cite{MR897431} virtually embeds into the right-angled Artin group defined by a path of length three. Here, we see that taking a finite-index subgroup is necessary as $\mathbb{Z}^\rtimes$ does not embed directly into a right-angled Artin group. Indeed, $\mathbb{Z}^\rtimes$ is two-generated but it is neither abelian nor free \cite{MR634562}. 
\end{ex}

\begin{ex}\label{HNNGP}
The previous example can be generalised in the following way. Consider a graph product $\Gamma \mathcal{G}$ and fix two vertices $u,v \in V(\Gamma)$ such that there exists an isomorphism $\varphi : G_u \to G_v$. The HNN extension $G:= \Gamma \mathcal{G} \ast_\varphi$ is a simple example of a fundamental group of right-angled graph of groups. Notice that $G$ contains a subgroup of index two $H$ which decomposes as a graph of groups with two vertices, both labelled by $\Gamma \mathcal{G}$; with two edges between these vertices, both labelled by $G_u$; such that one edge-group is sent into the first $\Gamma \mathcal{G}$ as $G_u$ and into the second $\Gamma \mathcal{G}$ as $G_v$ (through $\varphi$); and such that the second edge-group is sent into the first $\Gamma \mathcal{G}$ as $G_v$ (through $\varphi$) and into the second $\Gamma \mathcal{G}$ as $G_u$. Now Proposition \ref{prop:WhenSpecial} applies to $H$. Let $\Psi$ denote the graph obtained from two copies of $\Gamma$ by identifying $u,v$ in the first copy of $\Gamma$ respectively with $v,u$ in the second copy of $\Gamma$; and by adding a new neighbor to each of the two vertices in the intersection of the two copies of $\Gamma$. Also, let $\mathcal{H}$ denote the collection of groups indexed by $V(\Psi)$ such that a vertex $w$ of a copy of $\Gamma$ is labelled by $G_w \in \mathcal{G}$ and such that the two new vertices are labelled by $\mathbb{Z}_2$. According to Proposition \ref{prop:GraphForEmbedding}, our group $H$ embeds into $\Psi \mathcal{H}$.

\medskip \noindent
Thus, $\Gamma \mathcal{G} \ast_\varphi$ has a subgroup of index two which embeds (as a virtual retract if $\Gamma$ is finite) into the graph product $\Psi \mathcal{H}$. 

\medskip \noindent
For instance, the HNN extension $$G_{p,q} = \langle t, x_i \ (0 \leq i \leq p-1) \mid tx_0t^{-1}= x_2, x_i^q=[x_i,x_{i+1}]=1 \ (\text{$i$ mod $p$}) \rangle$$ of the Bourdon group $\Gamma_{p,q}$ \cite{MR1445387} has a subgroup of index two which embeds as a convex-cocompact subgroup into the graph product illustrated by Figure \ref{GP} for $p=5$. As an application, it can be deduced from \cite[Theorem 8.35]{Qm} and \cite[Theorem~1.8]{MR2153979} that $G_{p,q}$ is toral relatively hyperbolic. 
\end{ex}
\begin{figure}
\begin{center}
\includegraphics[scale=0.5]{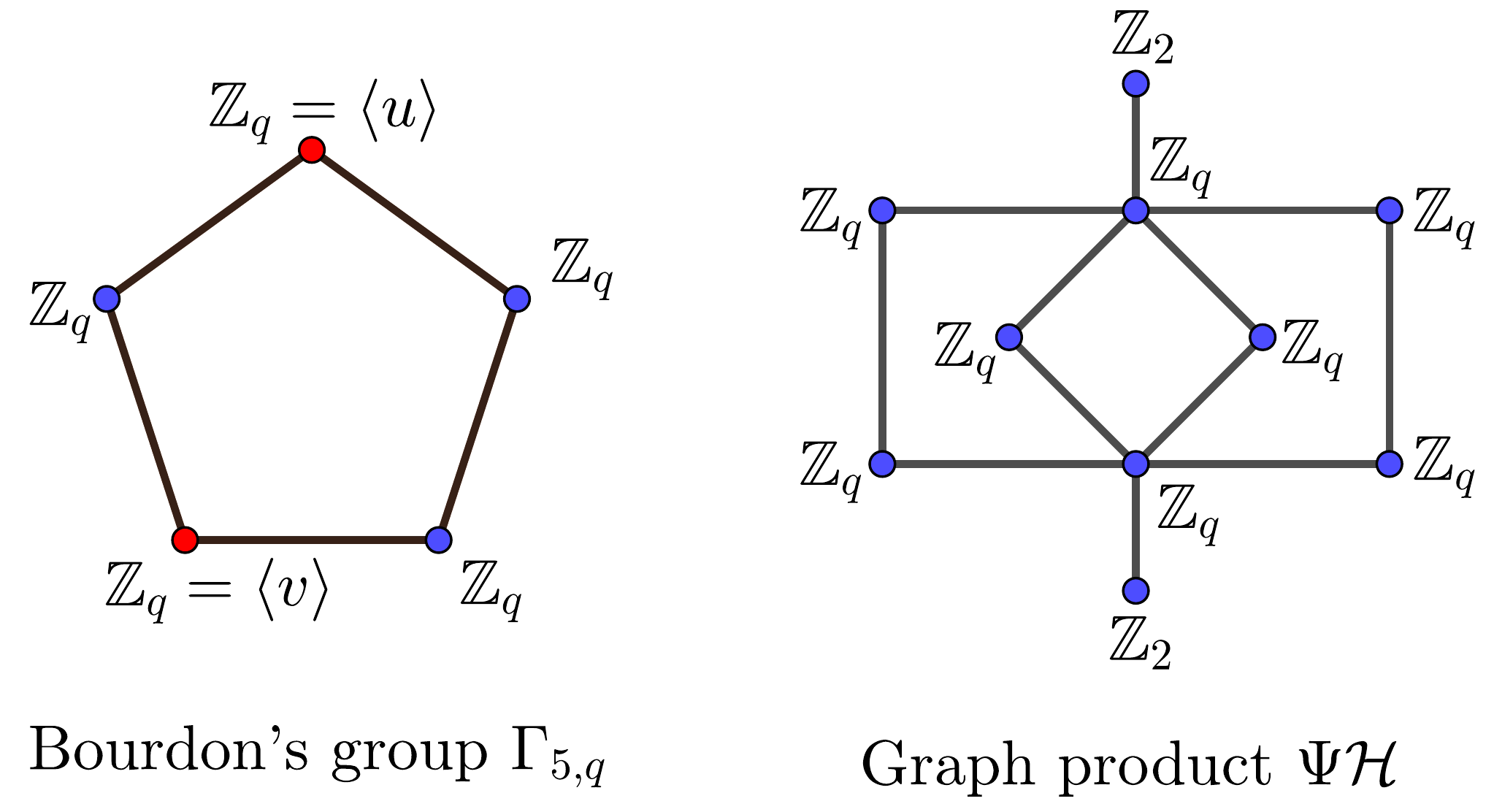}
\caption{The HNN extension $\Gamma_{5,q} \ast_{u^t=v}$ virtually embeds into $\Psi \mathcal{H}$.}
\label{GP}
\end{center}
\end{figure}

\begin{ex}
In our last example, we consider the group operation
$$G \bullet H = \langle G,H,t \mid [g,t^nht^{-n}]=1, \ g \in G, h \in H, n \geq 0 \rangle$$
introduced in \cite{MR1725439}. As observed in \cite{MR3868219}, $\mathbb{Z}\bullet \mathbb{Z}= \langle a,b,t \mid [a,t^nbt^{-n}]=1, n\geq 0 \rangle$ is a simple example of finitely generated but not finitely presented subgroup of $\mathbb{F}_2 \times \mathbb{F}_2$. We would like to generalise such an embedding for arbitrary factors.

\medskip \noindent
The product $G \bullet H$ can be decomposed as a right-angled graph of groups, since, given infinitely many copies $G_n,H_m$ of $G,H$ respectively ($n,m \in \mathbb{Z}$), it admits
$$\left\langle t, H_n,G_m, n,m \in \mathbb{Z} \left| \begin{array}{l} [g_{(n)},h_{(m)}]=1, n \geq m \\ tg_{(n)}t^{-1}=g_{(n+1)}, th_{(m)}t^{-1} = h_{(m+1)}, n,m \in \mathbb{Z} \end{array}, g \in G,h \in H \right. \right\rangle $$
as an alternative (relative) presentation, where $g_{(n)}$ (resp. $h_{(m)}$) denotes the element $g \in G$ in the copy $G_n$ (resp. the element $h \in H$ in the copy $H_m$). However, such a graph of groups (and each of its finite covers) does not satisfy Proposition \ref{prop:WhenSpecial}. So here we have an example of a fundamental group of a right-angled graph of groups for which the methods developed in the article do not work, even though nice embeddings exist (as sketched below).

\medskip \noindent
In order to embed $G \bullet H$ into a graph product, an alternative approach is to consider $G \bullet H$ as a diagram product \cite{MR1725439} and to look at its action on the quasi-median graph constructed in \cite{Qm}. We do not give details here, but the action turns out to be special, and an application of Proposition \ref{prop:GraphForEmbedding} shows that $G \bullet H$ embeds into $(G \ast \mathbb{Z}_2) \times (H \ast \mathbb{Z}_2)$ by sending $G$ to $G$, $H$ to $H$, and $t$ to $yx$ where $x$ (resp. $y$) is a non-trivial element of the left (resp. the right) $\mathbb{Z}_2$. (As a consequence of Remark \ref{remark:BiggerGroups}, the $\mathbb{Z}_2$ can be replaced with infinite cyclic groups, so that we recover the same embedding $\mathbb{Z} \bullet \mathbb{Z} \hookrightarrow \mathbb{F}_2 \times \mathbb{F}_2$ found in \cite{MR3868219}.)
\end{ex}

\addcontentsline{toc}{section}{References}

\bibliographystyle{alpha}
{\footnotesize\bibliography{QMspecial}}

\begin{flushleft}
Anthony Genevois\\
Universit\'e Paris-Saclay, Laboratoire de math\'ematiques d'Orsay, 91405, Orsay, France \\
\emph{e-mail:}\texttt{anthony.genevois@universite-paris-saclay.fr}\\[5mm]
\end{flushleft}

\end{document}